\documentclass[final,hidelinks,onefignum,onetabnum]{siamart220329}

\usepackage{amscd,amssymb, amsmath,amsfonts, wasysym, mathrsfs, enumerate, mathtools,hhline,xcolor}%
\usepackage{graphicx}
\usepackage[all, cmtip]{xy}
\usepackage{multirow}

\definecolor{hot}{RGB}{65,105,225}
\usepackage{algorithm}
\usepackage[noend]{algpseudocode}

\usepackage{hyperref}
\hypersetup{
	plainpages=false,
    colorlinks,
    linktocpage=true,
    linkcolor=hot,
    citecolor=hot,
    urlcolor=hot
}
\usepackage[hyperpageref]{backref}

\usepackage[colorinlistoftodos,prependcaption,textsize=footnotesize]{todonotes}

\makeatletter
\@mparswitchfalse%
\makeatother
\normalmarginpar %

\newtheorem{remark}[theorem]{Remark}
\newtheorem{assumption}[theorem]{Assumption}
\newtheorem{example}[theorem]{Example}

\begin{document}

\title{The geometric convergence of a parallel domain decomposition method on manifolds}

\author{Lizhen Qin \thanks{School of Mathematics, Nanjing University, Nanjing, Jiangsu, China
  (\email{qinlz@nju.edu.cn}).}
}

\maketitle

\headers{DDM on Manifolds}{L. Qin}

\begin{abstract}
This paper establishes a convergence theory for a continuous domain decomposition method for elliptic equations on manifolds. This method originated in \cite{lions2} in the setting of Euclidean domains and was later adapted and generalized to manifolds by \cite{qin_wang_wang}. Although its convergence was known, whether the convergence is geometric has remained open. We prove its geometric convergence and provide various estimates on the convergence rate.
\end{abstract}

\begin{AMS}{Primary 65N55; Secondary 58J05.}
\end{AMS}

\begin{keywords}{Riemannian manifolds, elliptic problems, domain decomposition methods, geometric convergence, convergence rate}
\end{keywords}

\section{Introduction}\label{sec_introduction}
Elliptic equations on Riemannian manifolds play an important role in both analysis and geometry (see e.g.,~\cite{Aubin,schoen_yau}).
They arise in many areas, such as image processing, multifluid dynamics, micromagnetics, and theoretical physics (see e.g.,~\cite{bachini2021intrinsic,bonito2020divergence,dobrev2010surface,Holst2018,Holst2016Wave,Holst_Stern,jankuhn2021error,jin2021gradient,mohamed2005finite,reuter2009discrete}). This paper presents an in-depth convergence analysis of a domain decomposition method (DDM) for second-order linear elliptic equations on manifolds at the continuous level. This DDM, formulated in Algorithm \ref{alg_continuous_parallel} below, was originally proposed in \cite[Algorithm~2.1]{qin_wang_wang} and has been numerically implemented in \cite{qin_wang_wang,jiang_qin_wang}.

When the target manifold $M$ is a two-dimensional Riemannian submanifold of $\mathbb{R}^{3}$, i.e.,~a surface, the numerical methods to solve differential equations on $M$ have been extensively studied over many years (see e.g.,~\cite{baumgardner_frederickson,dziuk88,dziuk91,nedelec,nedelec_planchard}). A conventional and popular approach is to solve the equations by finite element methods (FEMs) based on a global grid of $M$. Such a grid can be obtained by polyhedral approximation in $\mathbb{R}^{3}$. This approach has reached a high level of maturity and has been widely applied (see e.g.,~\cite{BDN20,DDE05,dziuk_elliott} for surveys and bibliographies).

Alternatively, Qin-Zhang-Zhang in \cite{qin_zhang_zhang} proposed an idea for solving elliptic equations on manifolds using the framework of DDMs. Since a $d$-dimensional manifold $M$ has local coordinate charts by definition, $M$ can be decomposed into finitely many subdomains that carry local coordinates. An elliptic equation in each subdomain can then be transformed to one in a domain in $\mathbb{R}^{d}$. Therefore, an elliptic problem on $M$ can be solved by DDMs with subproblems posed on Euclidean domains, provided that the DDMs are convergent.

The idea of \cite{qin_zhang_zhang} was further developed by Cao-Qin in \cite{cao_qin} and by Qin-Wang-Wang in \cite{qin_wang_wang}. Those works combine the philosophy of \cite{qin_zhang_zhang} with the seminal works of P.~L.~Lions. To solve elliptic equations in Euclidean domains, Lions proposed a sequential DDM in \cite[Section~I.~4]{lions1} and a parallel DDM in \cite[p.~66]{lions2}. Both methods are overlapping DDMs formulated at the continuous level. The continuous sequential DDM was extended to manifolds in \cite{cao_qin}; the parallel one was adapted, by introducing a partition of unity, and generalized to manifolds in \cite{qin_wang_wang}. At the numerical level, the aforementioned continuous DDMs were discretized using FEMs in \cite{cao_qin,qin_wang_wang}. More recently, Jiang-Qin-Wang further proposed numerical imitations of the continuous DDMs based on physics-informed neural network methods (PINNs) in \cite{jiang_qin_wang}. These numerical DDMs have been tested in \cite{cao_qin,qin_wang_wang,jiang_qin_wang} on various manifolds, both with and without boundary, in dimensions ranging from $4$ to $10$. The numerical results indicate that the DDM approaches perform effectively on high dimensional manifolds, including complicated ones that are not submanifolds of Euclidean spaces.

Compared with numerical implementations, the analysis of convergence is considerably more challenging. To date, a convergence theory for the aforementioned numerical DDMs has not been established. At the continuous level, the geometric convergence (see Definition \ref{def_rate} below) of the sequential DDM was proved by Lions in \cite[Theorem.~I.2]{lions1} for Euclidean domains, and his theorem and proof were extended straightforwardly to manifolds in \cite{cao_qin}. Thus, a satisfactory convergence theory for the continuous sequential DDM on manifolds is already available.

For the continuous parallel DDM, the situation is quite different. Lions also proved its convergence in \cite[p.~66-67]{lions2} for Euclidean domains. This proof, presented in a compressed style, relies on the technique of sub- and super-solutions and is much more elaborate than that in \cite{lions1}. Later, this proof was adapted and generalized to manifolds in \cite[Theorem~2.4]{qin_wang_wang}. However, neither \cite{lions2} nor \cite{qin_wang_wang} proved geometric convergence for Euclidean domains, let alone for manifolds. In fact, whether this DDM converges geometrically has remained open to this day.

In this paper, we prove that the continuous parallel DDM, i.e.,~Algorithm \ref{alg_continuous_parallel} below, does converge geometrically on manifolds. It's worth pointing out that, although this algorithm, proposed in \cite{qin_wang_wang}, is significantly inspired by \cite{lions2}, it differs from Lions' original method even in the special case where the manifold is a Euclidean domain: it employs a partition of unity which was not used in \cite{lions2} (see $\S$\ref{subsec_algorithm} for more details). By exploiting this partition of unity together with the classical maximum principle, we now establish the desired geometric convergence. Our proof differs substantially from those in \cite[p.~66-67]{lions2} and \cite{qin_wang_wang} (cf.~Remark \ref{rmk_qww_proof}).

Here is a brief description of the main results of this paper. We can indeed prove a very general Theorem \ref{thm_converge} on the geometric convergence in $C^{0}$-norm. However, its formulation and proof are technically rather involved. To fix ideas and enhance the readability of this paper, we shall first present detailed proofs of some special cases of Theorem \ref{thm_converge}. In these special cases, Theorem \ref{thm_converge} splits into three results: Corollary \ref{cor_rate_nondegenerate}, Theorems \ref{thm_rate_boundary} and \ref{thm_slow}. The proof of Theorem \ref{thm_rate_boundary} is particularly intricate and relies in part on some ideas from \cite{qin_xu06}. More precisely, we introduce a notion of winding number (see Definition \ref{def_winding} below) that is similar to, yet distinct from, that in \cite[Definition~4.1]{qin_xu06}. Following \cite{qin_xu06}, we apply an inductive argument on a sequence of subdomains, where the number of inductive steps is limited by the winding number. This type of argument originated in another work of Lions \cite{lions3} and was subsequently developed in \cite{deng2}, \cite{deng3}, and \cite{qin_xu06}.

Rather than providing a complete proof of Theorem \ref{thm_converge}, we present a sketch in $\S$\ref{subsec_converge_general}. The proof idea is parallel to that of the special cases mentioned above. Once the reader is familiar with the proofs of those special cases, the details of the general proof can be filled in along the lines sketched there.

In addition to establishing the geometric convergence in $C^{0}$-norm, we also derive several estimates on the convergence rate (see Definition \ref{def_rate} below). These results are stated in Theorems \ref{thm_bound}, \ref{thm_comparison}, and \ref{thm_refine}, which elucidate the manner in which the geometric structure of the domain decomposition and the coefficients of the differential equation influence the convergence. Finally, Theorems \ref{thm_high_global} and \ref{thm_high_local} demonstrate that $C^{0}$-convergence implies convergence of high-order derivatives, provided that the relevant regularity is assumed.

DDMs have a long history, originating with the alternating method invented by H.~A.~Schwarz \cite{schwarz}. The seminal works \cite{lions1,lions2,lions3} of P.~L.~Lions are natural and remarkable extensions of Schwarz's original idea. Since then, the field has grown greatly (see e.g.,~\cite{BPWX,CDS1999,dolean2015DDM,quarteroni_valli,smith_bjorstad_gropp,toselli_widlund,xu92,xu_zou}). In recent years, neural networks have emerged as an exceptionally powerful numerical tool for solving partial differential equations (see e.g.,~\cite{E_Yu2018} and \cite{RPK2019}.) Numerous works have combined DDM frameworks with neural network approaches (see e.g.,~\cite{DHMM2024,KKKK2024,sun_xu_yi_2024}); in particular, \cite{jiang_qin_wang} integrated the manifold-based DDM frameworks discussed above with PINNs.

As noted above, DDMs fall into two classes: continuous-level and numerical-level. Lions' work \cite{lions1,lions2,lions3} and the method under investigation in this paper belong to the former class. The latter class, by contrast, is far more represented in the literature and is undoubtedly of great importance for practical computation. The continuous formulations and their accompanying convergence theories, however, retain lasting significance. While DDMs have thus far been implemented numerically using FEMs and neural networks, future developments will definitely introduce novel numerical techniques that we cannot presently anticipate. With a continuous DDM and a satisfactory convergence theory in hand, it's promising to implement this DDM using any effective numerical approach.

The outline of this paper is as follows. In $\S$\ref{sec_main}, we first introduce the elliptic boundary value problem and the DDM under consideration, and then state our main results, which comprise a series of theorems. The proofs of these theorems are presented in the sections from $\S$\ref{sec_geometry} to $\S$\ref{sec_high}. In $\S$\ref{sec_general}, we discuss some generalizations of the main results; in particular, the most general Theorem \ref{thm_converge} on $C^{0}$-convergence is formulated and proved there. Finally, we summarize this paper in $\S$\ref{sec_conclusion}.

\section{Main Results}\label{sec_main}
In this section, we introduce a linear elliptic problem on a compact Riemannian manifold and recall a domain decomposition method (DDM) for solving this problem. This DDM was proposed at the continuous level in \cite[Algorithm~2.1]{qin_wang_wang}, and had been numerically validated in \cite{qin_wang_wang} and \cite{jiang_qin_wang} by finite element methods and physics-informed neural network methods, respectively. We now formulate a family of Theorems \ref{thm_bound}, \ref{thm_rate}, \ref{thm_rate_boundary}, \ref{thm_slow}, \ref{thm_comparison}, \ref{thm_refine}, \ref{thm_high_global}, and \ref{thm_high_local} concerning the convergence of this DDM. These theorems constitute the main results of this paper. Their proofs occupy the subsequent sections from $\S$\ref{sec_geometry} through $\S$\ref{sec_high}.

\subsection{Elliptic Problem}
Let $M$ be a $d$-dimensional compact smooth manifold without or with boundary $\partial M$. Equip $M$ with a $C^{\infty}$ Riemannian metric $g$. With respect to $g$, the Laplace operator $\Delta$, also named the Laplace-Beltrami operator, is defined on $M$. Note that neither $g$ nor $\Delta$ can be expressed by coordinates globally in general because $M$ does not necessarily have a global coordinate chart. In a local chart with coordinates $(x_{1}, \dots, x_{d})$, the Riemannian metric tensor $g$ is expressed as
\begin{equation}\label{eqn_metric}
g= \sum_{\alpha, \beta=1}^{d} g_{\alpha \beta} \mathrm{d} x_{\alpha} \otimes \mathrm{d} x_{\beta},
\end{equation}
where the matrix function $(g_{\alpha \beta})_{d \times d}$ is $C^{\infty}$ and its values are symmetric and positive definite. The Laplace operator $\Delta$ is expressed as
\begin{align}\label{eqn_laplace}
\Delta u & = \frac{1}{\sqrt{G}}\sum_{\alpha=1}^{d}\frac{\partial}{\partial x_{\alpha}} \left( \sum_{\beta=1}^{d}g^{\alpha \beta} \sqrt{G} \frac{\partial u}{\partial x_{\beta}} \right) \nonumber \\
& = \sum_{\alpha,\beta=1}^{d} g^{\alpha \beta}  \frac{\partial^{2} u}{\partial x_{\alpha} \partial x_{\beta}} + \frac{1}{\sqrt{G}} \sum_{\beta=1}^{d} \sum_{\alpha=1}^{d} \frac{\partial}{\partial x_{\alpha}} \left( g^{\alpha \beta} \sqrt{G} \right) \frac{\partial u}{\partial x_{\beta}},
\end{align}
where $G = \det \left( (g_{\alpha \beta})_{d \times d} \right)$ is the determinant of the matrix $(g_{\alpha \beta})_{d \times d}$ and $(g^{\alpha \beta})_{d \times d}$ is the inverse of $(g_{\alpha \beta})_{d \times d}$.

Define a differential operator $\mathfrak{L}$ on $M$ as
\begin{equation}\label{eqn_operator}
\mathfrak{L} u := - \Delta u + \langle \vec{b}, \nabla u \rangle + c u,
\end{equation}
where $\vec{b}$ is a $C^{\infty}$ vector field on $M$, the inner product $\langle \cdot, \cdot \rangle$ is given by $g$, and $c$ is a $C^{\infty}$ function on $M$. In the above local chart, we have
\begin{equation}\label{eqn_operator_local}
\mathfrak{L} u = - \sum_{\alpha,\beta=1}^{d} g^{\alpha \beta}  \frac{\partial^{2} u}{\partial x_{\alpha} \partial x_{\beta}} + \sum_{\alpha =1}^{d} b^{\alpha} \frac{\partial u}{\partial x_{\alpha}} + cu,
\end{equation}
where each $b^{\alpha}$ comes from the coefficients in the second summand in \eqref{eqn_laplace} and the local expression of $\vec{b}$. Therefore, each $b^{\alpha}$ is a $C^{\infty}$ function in this local chart. We further infer $\mathfrak{L}$ is a second-order linear elliptic differential operator with $C^{\infty}$ coefficients.

Since the Laplace operator $\Delta$ appearing in \eqref{eqn_operator} depends on $g$, we shall denote it by $\Delta_{g}$ when clarity is needed. Although $\Delta_{g}$ is a special operator, the operator \eqref{eqn_operator} is in fact universal for second-order linear elliptic operators on manifolds. Specifically, every such a general operator on a Riemannian manifold $(M,g)$ admits a reduction to the form \eqref{eqn_operator} by choosing a suitable Riemannian metric $\tilde{g}$ (which may differ from $g$), in which the Laplace operator is taken with respect to $\tilde{g}$. Consequently, the results of this paper apply to the full class of second-order linear elliptic equations on manifolds. For a detailed explanation, see $\S$\ref{subsec_type}.

Let $u \in C^{0} (M)$, define $f := \mathfrak{L} u$ as a distribution on $M$. Since the equation $\mathfrak{L} u = f$ can be studied on each component of $M$, we assume $M$ is connected throughout this paper. Define $\varphi := u|_{\partial M} \in C^{0} (\partial M)$ if $\partial M \ne \emptyset$. Then $u$ solves the following problem
\begin{equation}\label{eqn_problem}
\left\{
\begin{array}{rcl}
\mathfrak{L} u & = & f, \\
u|_{\partial M} & = & \varphi.
\end{array}
\right.
\end{equation}
Here, if $\partial M = \emptyset$, then $\varphi$ is undefined and the boundary value condition $u|_{\partial M} = \varphi$ is vacuously satisfied.

Though \eqref{eqn_problem} always has a solution $u$, it is not necessarily well-posed since its solution may not be unique. For example, if $\partial M = \emptyset$ and $c=0$ in \eqref{eqn_operator}, then $u+C$ is also a solution to \eqref{eqn_problem} for any constant $C$. Alternatively, suppose $\partial M \ne \emptyset$, $\varphi =0$, $\vec{b} =0$, and $c = \lambda <0$, where $\lambda$ is an eigenvalue of $\Delta$, then $u+ \psi$ is also a solution for any eigenfunction $\psi$ associated with $\lambda$. To ensure the well-posedness, we make the following assumption throughout this paper.
\begin{assumption}\label{asp_wellpose}
The manifold $M$ and the function $c$ in \eqref{eqn_operator} satisfy one of the following conditions:
\begin{enumerate}[(1)]
\item $c \ge 0$ and $c \ne 0$;

\item $c \ge 0$ and $\partial M \ne \emptyset$.
\end{enumerate}
\end{assumption}

\begin{lemma}\label{lem_wellpose}
Under Assumption \ref{asp_wellpose}, in the space $C^{0} (M)$, the problem \eqref{eqn_problem} has the unique distributional solution $u$.
\end{lemma}
\begin{proof}
Suppose $v$ is another solution, then $u-v \in C^{0} (M)$ and $\mathfrak{L} (u-v) =0$ in the distributional sense. By \cite[Theorem~7.4.1]{Hormander}, we have $(u - v) \in C^{\infty} (M \setminus \partial M)$. Then $\mathfrak{L} (u-v) =0$ in the classical sense.

In the case of the (2) in Assumption \ref{asp_wellpose}, since $\partial M \ne \emptyset$, we have $(u-v)|_{\partial M} =0$, then $u=v$ by the weak maximum principle \ref{cor_maximum}.

Now let's assume the (1) in Assumption \ref{asp_wellpose}. It remains to consider the case of $\partial M = \emptyset$. Since $M \setminus \partial M = M$ is compact, $u-v$ reaches its maximum and minimum in $M \setminus \partial M$. By the strong maximum principle \ref{cor_maximum}, we see $u-v$ is constant. Since $c \ne 0$ by assumption, we infer $u=v$.
\end{proof}

\subsection{Domain Decomposition Method}\label{subsec_algorithm}
Let's recall the domain decomposition method (DDM) in \cite[Algorithm~2.1]{qin_wang_wang} for solving \eqref{eqn_problem}.

Suppose $M$ is decomposed into $m$ subdomains, i.e.,~$M = \bigcup_{i=1}^{m} M_{i}$, such that $M_{i} \ne M$ for all $i$. Here we call $M_{i}$ a subdomain of $M$ if it is a connected, compact, topologically embedded submanifold with boundary and of codimension $0$. In particular, if $\Omega$ is an open subset of $\mathbb{R}^{d}$, then any $d$-dimensional finite polyhedron insider $\Omega$ is a subdomain of $\Omega$ in our sense. For a detailed discussion of subdomains, we refer to $\S$\ref{subsec_subdomain}.

Let $\partial M_{i}$ denote the manifold boundary of $M_{i}$. (There are two inequivalent definitions of a boundary, one is in the setting of manifold theory, another is in point-set topology. For a detailed explanation on interiors and boundaries, see $\S$\ref{subsec_boundary}.) Let $\gamma_{i} = \partial M_{i} \setminus \partial M$. We see $\overline{\gamma_{i}} \subseteq \partial M_{i}$, and $M_{i} \setminus \overline{\gamma_{i}}$ is the interior of $M_{i}$ in the sense of point-set topology (see Lemma \ref{lem_subdomain_interior} below), where $\overline{\gamma_{i}}$ is the closure of $\gamma_{i}$ in $M$. If $M_{i} \cap \partial M = \emptyset$, particularly $\partial M = \emptyset$, then $\overline{\gamma_{i}} = \gamma_{i} = \partial M_{i}$.

If for each $x \in \partial M_{i}$, there is a coordinate chart $U$ containing $x$ such that, under this specific coordinate system, $U \cap \partial M_{i}$ is Lipschitz in the usual sense in $\mathbb{R}^{d}$ (see \cite[4.9]{Adams_Fournier}), then we say $\partial M_{i}$ is Lipschitz. As proved in $\S$\ref{subsec_lipschitz}, this definition of Lipschitz boundary does not depend on the particular choice of charts. Consequently, if $\partial M_{i}$ is Lipschitz near $x$ in some chart, then so is it in any chart containing $x$.

We make the assumption below on decompositions throughout this paper. By a partition of unity, we mean each $\rho_{i}$ is a nonnegative function and $\sum_{i=1}^{m} \rho_{i} =1$.
\begin{assumption}\label{asp_decomposition}
Suppose $M = \bigcup_{i=1}^{m} (M_{i} \setminus \overline{\gamma_{i}})$, and each $\partial M_{i}$ is Lipschitz. There is a partition of unity $\rho := \{ \rho_{i} \mid 1 \leq i \leq m \}$ subordinate to this decomposition, i.e.,~$\emptyset \ne \mathrm{supp} \rho_{i} \subset M_{i} \setminus \overline{\gamma_{i}}$, where $\mathrm{supp} \rho_{i}$ is the support of $\rho_{i}$. Furthermore, each $\rho_{i}$ is continuous.
\end{assumption}

By Assumption \ref{asp_decomposition} and Lemma \ref{lem_subdomain_interior}, the $\{ M_{i} \setminus \overline{\gamma_{i}} \mid 1 \le i \le m \}$ is an open cover of $M$. Thus our decomposition is overlapping, and one can always find a $C^{\infty}$ partition of unity subordinate to this open cover (see \cite[1.11]{Warner}). Here by a $C^{\infty}$ partition, we mean each $\rho_{i}$ is $C^{\infty}$. The following Algorithm \ref{alg_continuous_parallel} is the DDM under investigation.
\begin{algorithm}
\caption{(\cite{qin_wang_wang})~DDM to solve \eqref{eqn_problem}, under Assumptions \ref{asp_wellpose} and \ref{asp_decomposition}.}
\label{alg_continuous_parallel}

\begin{algorithmic}[1]
\State%
Choose an initial guess function $u^{0}$ on $M$ with $u^{0}|_{\partial M} = \varphi$.

\State%
For each $n>0$,
\begin{quote}
for $1 \leq i \leq m$,
\begin{quote}
find a function $u^{n}_{i}$ on $M_{i}$ such that
\begin{equation}\label{alg_continuous_parallel_1}
\left\{
\begin{aligned}
\mathfrak{L} u^{n}_{i}  & = f, & \text{in $M_{i} \setminus \partial M_{i}$;}
\\
u^{n}_{i}  & = u^{n-1}, & \text{on $\partial M_{i}$.}
\end{aligned}
\right.
\end{equation}
\end{quote}
\end{quote}

Let $u^{n} = \sum_{i=1}^{m} \rho_{i} u^{n}_{i}$.
\end{algorithmic}
\end{algorithm}

We first point out that Algorithm \ref{alg_continuous_parallel} is well-posed. Though $u^{n}_{i}$ in \eqref{alg_continuous_parallel_1} is merely defined on $M_{i}$, the subsequent $u^{n}$ is well-defined on the whole $M$ since $\mathrm{supp} \rho_{i} \subset M_{i}$. More precisely, the function $\rho_{i} u^{n}_{i}$ on $M_{i}$ is extended as $0$ outside of $M_{i}$.
\begin{lemma}\label{lem_algorithm_wellpose}
Suppose $u^{0} \in C^{0} (M)$ in Algorithm \ref{alg_continuous_parallel}. Then, $\forall n>0$, in the space $C^{0} (M_{i})$, the \eqref{alg_continuous_parallel_1} has a unique distributional solution $u^{n}_{i}$ for each $i$, and $u^{n} \in C^{0} (M)$. Furthermore, $u^{n}_{i}|_{\partial M} = \varphi$ and $u^{n}|_{\partial M} = \varphi$.
\end{lemma}
\begin{proof}
Trivially, $u^{n}_{i}|_{\partial M} = \varphi$ and $u^{n}|_{\partial M} = \varphi$ if \eqref{alg_continuous_parallel_1} is solvable. The proof of the uniqueness of $u^{n}_{i}$ duplicates that of Lemma \ref{lem_wellpose} under the (2) in Assumption \ref{asp_wellpose}.

It remains to prove the existence of $u^{n}_{i}$ and the continuity of $u^{n}$. Let's consider the problem
\[
\left\{
\begin{aligned}
\mathfrak{L} v  & = 0, & \text{in} \ M_{i} \setminus \partial M_{i};
\\
v & = u^{0} -u, & \text{on} \ \partial M_{i}.
\end{aligned}
\right.
\]
By the assumption, both $u^{0}$ and $u$ are in $C^{0} (M)$, and $\partial M_{i}$ is Lipschitz. By Lemma \ref{lem_perron_lipschitz}, the above problem has a Perron solution $v \in C^{0} (M_{i})$. Now $u^{1}_{i} := v+u$ is a solution to \eqref{alg_continuous_parallel_1} in $C^{0} (M_{i})$ for $n=1$. Since $u^{1}_{i}$ and $\rho_{i}$ are continuous for all $i$, we infer $u^{1} \in C^{0} (M)$ (see also the proof of Lemma \ref{lem_error_operator}).

We finish the proof by induction on $n$.
\end{proof}

Originally, Lions proposed a DDM in \cite[p.~66]{lions2} for elliptic equations on Euclidean domains. This DDM was adapted and generalized to Algorithm \ref{alg_continuous_parallel} by \cite{qin_wang_wang}. Though Algorithm \ref{alg_continuous_parallel} follows \cite{lions2} largely, it differs from \cite{lions2} even in the special case where the target manifold $M$ is a Euclidean domain. Actually, the above partition of unity was not used in \cite{lions2}. Instead, the boundary value $u^{n}_{i}|_{\gamma_{i}}$ in \cite{lions2} was chosen as an arbitrary function $v$ such that (see \cite[(42)~\&~(43)]{lions2})
\begin{equation}\label{eqn_lions}
\min \{ u^{n-1}_{j} (x) \mid x \in M_{j}, j \neq i \} \leq v(x) \leq \max \{ u^{n-1}_{j} (x) \mid x \in M_{j}, j \neq i \}.
\end{equation}
Lions then proved that his $u^{n}_{i}$ converges uniformly to $u$ on each compact set $K_{i} \subset M_{i} \setminus \overline{\gamma_{i}}$, but the geometric convergence was not established in any sense.

Our $u^{n}_{i}|_{\gamma_{i}}$ in Algorithm \ref{alg_continuous_parallel} certainly satisfies \eqref{eqn_lions}. Beyond this, thanks to the partition of unity in Assumption \ref{asp_decomposition}, the iteration in Algorithm \ref{alg_continuous_parallel} even has a better theory of convergence. First, our $u^{n}_{i}$ is continuous by Lemma \ref{lem_algorithm_wellpose}, whereas the $u^{n}_{i}$ in \cite{lions2} was usually not necessarily continuous on $M_{i}$ since $u^{n}_{i}|_{\gamma_{i}}$ was not necessarily continuous. Second, the partition of unity allows us to prove the geometric convergence of Algorithm \ref{alg_continuous_parallel}.

\begin{remark}\label{rmk_qww_proof}
The proof in \cite{qin_wang_wang} also exploited the partition of unity to show $u^{n}_{i}$ converges uniformly to $u$ on $M_{i}$. This result is slightly stronger than Lions', see also \cite[Remark~2.5]{qin_wang_wang}. However, since that proof largely follows \cite[p.~66-67]{lions2}, a desired geometric convergence was not obtained.
\end{remark}

\subsection{Convergence Rate: Upper Bound}
Let $\mathcal{D} := \{ M_{1}, \dots, M_{m} \}$ and $\rho := \{ \rho_{1}, \dots, \rho_{m} \}$ denote the domain decomposition and the partition of unity, respectively, in Assumption \ref{asp_decomposition}. We now define a constant $\Theta$. Theorem \ref{thm_rate} below indicates that $\Theta$ is an upper bound on the convergence rate of Algorithm \ref{alg_continuous_parallel}.

Let $\theta_{i}$ be the Perron solution (see $\S$\ref{subsec_Perron}) to the problem
\begin{equation}\label{eqn_bound_function}
\left\{
\begin{aligned}
\mathfrak{L} \theta_{i} & = 0, & \text{in $M_{i} \setminus \partial M_{i}$}, \\
\theta_{i} & = 1, & \text{on $\overline{\gamma_{i}}$}, \\
\theta_{i} & = 0, & \text{on $\partial M_{i} \setminus \overline{\gamma_{i}}$}.
\end{aligned}
\right.
\end{equation}
Let $\Theta_{i} := \sup_{\mathrm{supp} \rho_{i}} \theta_{i}$. Let $\Theta := \max \{ \Theta_{1}, \dots, \Theta_{m} \}$.

\begin{theorem}\label{thm_bound}
We have $0 < \Theta \le 1$. If in \eqref{eqn_operator}, $c>0$ somewhere in $M_{i}$ for all $i$ satisfying $\partial M_{i} = \overline{\gamma_{i}}$, then we further have $\Theta <1$.
\end{theorem}

\begin{remark}
The formulation and proof of Theorem \ref{thm_bound} are motivated by the Proposition 2 and Lemma 3 in \cite{lions2}. That Proposition 2 is actually related to our Theorem \ref{thm_bound}. The two-subdomain decomposition in that proposition is a special case of ours. Let $k_{1}$ and $k_{2}$ be the ones in that proposition. Then we have $k_{i} \le \Theta_{i}$ for $i=1,2$, where $\Theta_{i}$ is the one in our Theorem \ref{thm_bound}.
\end{remark}

\subsection{Winding Number}
To prove the geometric convergence, a difficult situation is that $\partial M \ne \emptyset$ and $c=0$. We need to study the geometric feature of the decomposition. Define $U_{i} := \{ x \in M \mid \rho_{i} (x) >0 \}$. By Assumption \ref{asp_decomposition}, we have $\bigcup_{i=1}^{m} U_{i} = M$. Define $I_{0} := \emptyset$,
\[
I_{1} := \{ i \mid 1 \le i \le m, \partial M_{i} \ne \overline{\gamma_{i}} \}.
\]
Suppose $\partial M \ne \emptyset$ temporarily, we know $I_{1} \ne \emptyset$ by Lemma \ref{lem_winding_start} below. When $n>1$, we define inductively
\[
I_{n} := I_{n-1} \cup \left\{ i \middle| 1 \le i \le m, \partial M_{i} \cap \left( \bigcup_{j \in I_{n-1}} U_{j} \right) \ne \emptyset \right\}.
\]
(When clarity requires, we denote these sets by $U_{i} (\rho)$ and $I_{n} (\mathcal{D}, \rho)$, respectively.) By the definition, $I_{n} \subseteq I_{n+1}$ for each $n$. If $I_{n} = I_{n+1}$, then $I_{n} = I_{k}$ for all $k>n$. Since $i$ is limited from $1$ to $m$, we know there exists an $N \ge 1$ such that $I_{N-1} \ne I_{N} = I_{N+1}$.

We can prove the following result.
\begin{proposition}\label{prop_winding}
If $\partial M \ne \emptyset$, then $I_{N} = \{ i \mid 1 \le i \le m \}$.
\end{proposition}
\begin{definition}\label{def_winding}
We call the above $N$ the \emph{winding number} of $(\mathcal{D}, \rho)$. We also let $N(\mathcal{D}, \rho)$ denote it.
\end{definition}

The winding number in Definition \ref{def_winding} is similar to but different from that winding number in \cite[Definition~4.1]{qin_xu06}. There is one obvious difference. The domain decomposition in \cite{qin_xu06} is nonoverlapping, while it is overlapping here. Nonetheless, both winding numbers reflect, in some sense, the complexity of the decomposition: the higher the winding number, the more complicated the decomposition. This complexity in turn affects the convergence rate. Further discussion of winding numbers, together with concrete examples, can be found in $\S$\ref{subsec_winding}.

\subsection{$C^{0}$-Convergence}
Let's recall the definition of geometric convergence.
\begin{definition}\label{def_rate}
Suppose there exist constants $C \ge 0$ and $L \in [0,1)$ such that, $\forall u^{0} \in C^{0} (M)$ satisfying $u^{0}|_{\partial M} = \varphi$, $\forall n>0$, we have
\begin{equation}\label{def_rate_1}
\| u^{n} - u\|_{C^{0} (M)} \le C L^{n} \| u^{0} - u\|_{C^{0} (M)},
\end{equation}
where $u^{0}$ and $u^{n}$ are the ones in Algorithm \ref{alg_continuous_parallel}, and $C$ and $L$ are independent of $u^{0}$ and $n$. Then we say Algorithm \ref{alg_continuous_parallel} \emph{converges geometrically} in norm $\| \cdot \|_{C^{0} (M)}$ and call $L$ an \emph{upper bound on the convergence rate}.
\end{definition}

Many work in the literature call $L$ the \emph{convergence rate} directly (see e.g.,~\cite[4)~in~p.~52]{lions2}). This terminology is widely accepted. However, there are many pairs $(C,L)$ satisfying \eqref{def_rate_1}. For example, if $(C,L)$ satisfies \eqref{def_rate_1}, then so does $(C,L')$ for $L< L' <1$. In practice, one should try to find a desired $L$ as small as possible, which is nothing but a careful estimate of the convergence rate.

The following Theorem \ref{thm_rate} is a general estimate of convergence rate. Again, we make Assumptions \ref{asp_wellpose} and \ref{asp_decomposition} throughout this paper.
\begin{theorem}\label{thm_rate}
Let $u^{0} \in C^{0} (M)$ satisfying $u^{0}|_{\partial M} = \varphi$ be an initial guess of Algorithm \ref{alg_continuous_parallel}. Let $u^{n}$ and $u^{n}_{i}$ be the approximations generated by Algorithm \ref{alg_continuous_parallel}.

Then, $\forall n>0$, the following hold:
\begin{enumerate}[(1)]
\item $\| u^{n} - u \|_{C^{0} (M)} \le \max_{1 \le i \le m} \{ \| u^{n}_{i} - u \|_{C^{0} (M_{i})} \} \le \| u^{n-1} - u \|_{C^{0} (M)}$;

\item $\| u^{n} - u \|_{C^{0} (M)} \le \Theta \| u^{n-1} - u \|_{C^{0} (M)}$ and $\| u^{n} - u \|_{C^{0} (M)} \le \Theta^{n} \| u^{0} - u \|_{C^{0} (M)}$, where $\Theta$ is the one in Theorem \ref{thm_bound}.
\end{enumerate}
\end{theorem}

Theorems \ref{thm_bound} and \ref{thm_rate} immediately imply the following.
\begin{corollary}\label{cor_rate_nondegenerate}
Assuming further, in \eqref{eqn_operator}, the function $c>0$ somewhere in $M_{i}$ for all $i$ satisfying $\partial M_{i} = \overline{\gamma_{i}}$. (Particularly, $c>0$ on $M$.) Then Algorithm \ref{alg_continuous_parallel} converges geometrically, and $\| u^{n} - u \|_{C^{0} (M)} \le \Theta \| u^{n-1} - u \|_{C^{0} (M)}$ for all $n>0$ with the constant $\Theta <1$.
\end{corollary}

\begin{remark}
By (1) of Theorem \ref{thm_rate}, we see $\max_{1 \le i \le m} \{ \| u^{n}_{i} - u \|_{C^{0} (M_{i})} \}$ and $\| u^{n} - u \|_{C^{0} (M)}$ share the same upper bound of convergence rate. In fact, set $u^{0}_{i} = u^{0}|_{M_{i}}$ and suppose \eqref{def_rate_1} with $L>0$, then
\begin{align*}
\max_{1 \le i \le m} \{ \| u^{n}_{i} - u \|_{C^{0} (M_{i})} \} & \le \| u^{n-1} - u \|_{C^{0} (M)} \le C L^{n-1} \| u^{0} - u\|_{C^{0} (M)} \\
& = \tfrac{C}{L} L^{n} \| u^{0} - u\|_{C^{0} (M)} = \tfrac{C}{L} L^{n} \max_{1 \le i \le m} \{ \| u^{0}_{i} - u \|_{C^{0} (M_{i})} \},
\end{align*}
Similarly, $\max_{1 \le i \le m} \{ \| u^{n}_{i} - u \|_{C^{0} (M_{i})} \} \le CL^{n} \max_{1 \le i \le m} \{ \| u^{0}_{i} - u \|_{C^{0} (M_{i})} \}$ implies $\| u^{n} - u \|_{C^{0} (M)} \le C L^{n} \| u^{0} - u\|_{C^{0} (M)}$.
\end{remark}

In the case of (2) in Assumption \ref{asp_wellpose}, it can happen that $\Theta =1$. For example, $c=0$ and the winding number $N>1$. Fortunately, we can prove the following result.

\begin{theorem}\label{thm_rate_boundary}
Assuming further $\partial M \ne \emptyset$, let $N$ be the winding number in Definition \ref{def_winding}. Let $u^{0} \in C^{0} (M)$ satisfying $u^{0}|_{\partial M} = \varphi$ be an initial guess of Algorithm \ref{alg_continuous_parallel}. Let $u^{n}$ be the approximation generated by Algorithm \ref{alg_continuous_parallel}.

Then there exists a constant $L \in (0,1)$ such that: $\forall n \ge 0$,
\begin{enumerate}[(1)]
\item $\| u^{n+N} - u \|_{C^{0} (M)} \le L^{N} \| u^{n} - u \|_{C^{0} (M)}$;

\item there exists a constant $C>0$ such that
\[
\| u^{n} - u \|_{C^{0} (M)} \le C L^{n} \| u^{0} - u \|_{C^{0} (M)}.
\]
\end{enumerate}
Here $L$ and $C$ do not depend on $n$, $u^{0}$ and $u$.
\end{theorem}

The proof of Theorem \ref{thm_rate} is relatively straightforward. However, that of Theorem \ref{thm_rate_boundary} is much more complicated.

The (1) of Theorem \ref{thm_rate_boundary} suggests that the convergence would be slow if the winding number $N$ is large. The following theorem shows that this is indeed true in some special case.

\begin{theorem}\label{thm_slow}
Suppose $\partial M \ne \emptyset$, $N>2$, and $M \ne \bigcup_{i \ne j} M_{i}$ for each $j$. Suppose further $c=0$ in \eqref{eqn_operator}. Then there is a $u^{0} \in C^{0} (M)$ such that $u^{0}|_{\partial M} = \varphi$, $u^{0} \ne u$ and
\begin{equation}\label{thm_slow_1}
\| u^{n} - u \|_{C^{0} (M)} = \| u^{0} - u \|_{C^{0} (M)} \qquad \text{for $n<N-1$}.
\end{equation}
Assume additionally $M_{i} \cap \partial M = \emptyset$ for every $i$ satisfying $\partial M_{i} = \overline{\gamma_{i}}$, then the $u^{0}$ can be chosen such that \eqref{thm_slow_1} holds for $n < N$.
\end{theorem}

Theorem \ref{thm_slow} indicates that, in certain cases, the $C^{0}$-norms of the errors of iterative approximations do not decrease until step $N$; see also Remark \ref{rmk_slow}. This also shows the estimate in (1) of Theorem \ref{thm_rate_boundary} is sharp.

\begin{remark}
Theorem 5.1 in \cite{qin_xu06} also demonstrates that a higher winding number leads to slower convergence in that setting.
\end{remark}

The above results are insufficient, under Assumptions \ref{asp_wellpose} and \ref{asp_decomposition} alone, to establish the geometric convergence of Algorithm \ref{alg_continuous_parallel}. To see this, suppose $\partial M = \emptyset$ and $c>0$ somewhere, but $c=0$ in some subdomain $M_{i}$. Then Theorem \ref{thm_rate_boundary} does not apply. Furthermore, $\Theta =1$ since $\Theta_{i} =1$, so Theorem \ref{thm_rate} is of no help either. Nevertheless, Algorithm \ref{alg_continuous_parallel} does converge geometrically in full generality. Actually, there is a general Theorem \ref{thm_converge} which covers Corollary \ref{cor_rate_nondegenerate} and Theorems \ref{thm_rate_boundary} and \ref{thm_slow}, and ensures the desired geometric convergence. Although the idea of its proof parallels that of Theorem \ref{thm_rate_boundary}, its formulation and proof are technically more complicated. To fix ideas and enhance the readability of this paper, we shall refrain from giving a full proof and instead provide a brief exposition of the main idea in $\S$\ref{subsec_converge_general}.

\subsection{Convergence Rate: Comparison}
By Theorem \ref{thm_rate}, the $\Theta$ in Theorem \ref{thm_bound} is an upper bound on the convergence rate of Algorithm \ref{alg_continuous_parallel}. Since $\Theta$ depends on the function $c$, the decomposition $\mathcal{D}$ and the partition of unity $\rho$, we write it as $\Theta (c, \mathcal{D}, \rho)$. Our aim now is to investigate this dependence.

\begin{theorem}\label{thm_comparison}
Let $\tilde{c}$ be another $C^{\infty}$ function on $M$. Let $\widetilde{\mathcal{D}} := \{ \widetilde{M}_{1}, \dots, \widetilde{M}_{m} \}$ be another decomposition. Let $\widetilde{\gamma}_{i} := \partial \widetilde{M}_{i} \setminus \partial M$. Then the followings hold:
\begin{enumerate}[(1)]
\item If $\tilde{c} \ge c$, then $\Theta (\tilde{c}, \mathcal{D}, \rho) \le \Theta (c, \mathcal{D}, \rho)$.

\item If $\tilde{c} \ge c$, and $\tilde{c} >c$ somewhere in $M_{i}$ for all $i$ (particularly, $\tilde{c} >c$ on $M$), then $\Theta (\tilde{c}, \mathcal{D}, \rho) < \Theta (c, \mathcal{D}, \rho)$.

\item Suppose $c$ takes constant value, then $\lim\limits_{c \rightarrow +\infty} \Theta (c, \mathcal{D}, \rho) =0$.

\item If $\widetilde{M}_{i} \supseteq M_{i}$ for all $i$, then $\Theta (c, \widetilde{\mathcal{D}}, \rho) \le \Theta (c, \mathcal{D}, \rho)$.

\item If $\Theta (c, \mathcal{D}, \rho) <1$ and $\widetilde{M}_{i} \setminus \overline{\widetilde{\gamma}_{i}} \supseteq M_{i}$ for all $i$, then $\Theta (c, \widetilde{\mathcal{D}}, \rho) < \Theta (c, \mathcal{D}, \rho)$.
\end{enumerate}
\end{theorem}

\begin{remark}
Parts (1)-(3) of Theorem \ref{thm_comparison} indicate that the upper bound $\Theta$ on the convergence rate becomes better with increasing $c$. Parts (4)-(5) indicate the same with a larger overlap of the decomposition.
\end{remark}

\begin{definition}
Let $(\mathcal{D}, \rho)$ and $(\widetilde{\mathcal{D}}, \tilde{\rho})$ be two pairs consisting of a decomposition and a subordinate partition of unity. If $\forall M_{i} \in \mathcal{D}$, $\exists \{ \widetilde{M}_{i_{1}}, \dots, \widetilde{M}_{i_{k(i)}} \} \subseteq \widetilde{\mathcal{D}}$ such that
\[
M_{i} = \bigcup_{j=1}^{k(i)} \widetilde{M}_{i_{j}} \qquad \text{and} \qquad \mathrm{supp} \rho_{i} \subseteq \bigcup_{j=1}^{k(i)} \mathrm{supp} \tilde{\rho}_{i_{j}},
\]
then we say $(\widetilde{\mathcal{D}}, \tilde{\rho})$ is a \emph{refinement} of $(\mathcal{D}, \rho)$. If additionally, $\forall i$, $M_{i} \ne \widetilde{M}_{i_{j}}$ for $1 \le j \le k(i)$, then we say $(\widetilde{\mathcal{D}}, \tilde{\rho})$ is a \emph{strict refinement} of $(\mathcal{D}, \rho)$.
\end{definition}

\begin{theorem}\label{thm_refine}
Suppose $(\widetilde{\mathcal{D}}, \tilde{\rho})$ is a refinement of $(\mathcal{D}, \rho)$. Then:
\begin{enumerate}[(1)]
\item $\Theta (c, \mathcal{D}, \rho) \le \Theta (c, \widetilde{\mathcal{D}}, \tilde{\rho})$.

\item If the refinement is strict and $\Theta (c, \mathcal{D}, \rho) < 1$, then $\Theta (c, \mathcal{D}, \rho) < \Theta (c, \widetilde{\mathcal{D}}, \tilde{\rho})$.
\end{enumerate}
\end{theorem}

\begin{remark}
Theorem \ref{thm_refine} shows that refining the decomposition deteriorates the upper bound $\Theta$ on the convergence rate. This suggests that, to ensure a fast convergence, one should keep the number of subdomains small. In an extremal situation of a single subdomain, there is effectively no decomposition, and the solution is achieved in one iteration.
\end{remark}

\subsection{High Regularity}
We claim convergence in derivative norms under extra assumptions. For the definitions of $C^{k,\alpha}$- and $W^{k,p}$-norms on manifolds, see $\S$\ref{subsec_function}.
\begin{theorem}\label{thm_high_global}
Let $u^{0} \in C^{0} (M)$ satisfying $u^{0}|_{\partial M} = \varphi$ be an initial guess of Algorithm \ref{alg_continuous_parallel}. Suppose further $\rho_{i} \in C^{k, \alpha} (M)$ for all $i$, where $k \ge 0$ and $0 \le \alpha \le 1$. Then, for each $n>0$, we have $u^{n} -u \in C^{k, \alpha} (M)$ and
\begin{equation}\label{thm_high_global_1}
\| u^{n} -u \|_{C^{k, \alpha} (M)} \le C \max_{1 \le i \le m} \{ \| u^{n}_{i} - u \|_{C^{0}(M_{i})} \},
\end{equation}
where $C$ is a constant independent of $n$, $u$, and $u^{0}$.
\end{theorem}

Note that the high regularity of $u$ is even not assumed in Theorem \ref{thm_high_global}. Thus we obtain the following corollary.
\begin{corollary}\label{cor_high_global}
In Theorem \ref{thm_high_global}, assume further $u \in C^{k,\alpha} (M)$, then $u^{n} \in C^{k,\alpha} (M)$ for each $n>0$.
\end{corollary}

\begin{theorem}\label{thm_high_local}
Let $u^{0} \in C^{0} (M)$ satisfying $u^{0}|_{\partial M} = \varphi$ be an initial guess of Algorithm \ref{alg_continuous_parallel}. Then the following hold:
\begin{enumerate}[(1)]
\item If $\rho_{i} \in C^{0,1} (M)$ for all $i$, then for each $n>1$, we have $u^{n}_{i} -u \in H^{1} (M_{i})$ for all $i$, and
\begin{equation}\label{thm_high_local_1}
\max_{1 \le i \le m} \{ \| u^{n}_{i} - u \|_{H^{1}(M_{i})} \} \le C \max_{1 \le i \le m} \{ \| u^{n-1}_{i} - u \|_{C^{0}(M_{i})} \},
\end{equation}
where $C$ is a constant independent of $n$, $u$ and $u^{0}$.

\item Suppose further $\partial M_{j}$ is $C^{k,\alpha}$ for some $j$, where $k \ge 2$ and $0 < \alpha \le 1$. If $\rho_{i} \in C^{k,\alpha} (M)$ for all $i$, then for each $n>1$, we have $u^{n}_{j} -u \in C^{k,\alpha} (M_{j})$ and
\begin{equation}\label{thm_high_local_2}
\| u^{n}_{j} - u \|_{C^{k,\alpha}(M_{j})} \le C_{j} \max_{1 \le i \le m} \{ \| u^{n-1}_{i} - u \|_{C^{0}(M_{i})} \},
\end{equation}
where $C_{j}$ is a constant independent of $n$, $u$ and $u^{0}$.
\end{enumerate}
\end{theorem}

\begin{corollary}\label{cor_high_local}
In Theorem \ref{thm_high_local}, assume further $u \in H^{1} (M)$ in (1) or $u|_{M_{j}} \in C^{k,\alpha} (M_{j})$ in (2), then $u^{n}_{i} \in H^{1} (M_{i})$ or $u_{j}^{n} \in C^{k,\alpha} (M_{j})$, respectively, for each $n>1$.
\end{corollary}

\begin{remark}
By Rademacher's theorem (see \cite[\S4.2.3]{evans_gariepy}), in Theorem \ref{thm_high_global}, $u^{n} - u \in W^{k+1,\infty} (M)$ when $u^{n} - u \in C^{k,1} (M)$. Thus $\| u^{n} - u \|_{W^{k+1,p} (M)}$ is dominated by $\| u^{n} -u \|_{C^{k, 1} (M)}$ for $1 \le p \le \infty$. So $u^{n} \in W^{k+1,p} (M)$ if $u \in W^{k+1,p} (M)$ and $\rho_{i} \in C^{k,1} (M)$ for all $i$. Similar results hold in Theorem \ref{thm_high_local} and Corollary \ref{cor_high_local}.
\end{remark}

\section{Geometric Aspects of Domain Decomposition}\label{sec_geometry}
For numerical computation in a Euclidean domain $\Omega \subset \mathbb{R}^{d}$, one typically approximates $\Omega$ by a $d$-dimensional polyhedron and, when DDMs are employed, decomposes it into a collection of subpolyhedra. Unfortunately, this simple procedure cannot be carried over to a general manifold. Thus one needs to find a suitable definition of a subdomain in a manifold. Moreover, some basic notions related to domain decompositions, such as interior and boundary, admit multiple and inequivalent definitions in the manifold setting. The purpose of this section is to address these issues, to clarify the relevant concepts, and to prove some topological facts frequently used throughout this work. In particular, we prove Proposition \ref{prop_winding} which will be used in the proof of Theorem \ref{thm_rate_boundary}.

\subsection{Interior and Boundary}\label{subsec_boundary}
We need to clarify the terminology of ``interior" and ``boundary", as both have two distinct meanings: one in point-set topology and one in manifold theory. All are used in this paper.

Let's first consider point-set topology. Suppose $X$ is a topological space and $A \subseteq X$. We say $a \in X$ is an interior point of $A$ if $A$ is a neighborhood of $a$. We say $a \in X$ is a boundary point of $A$ if $U \cap A \ne \emptyset$ and $U \setminus A \ne \emptyset$ for all neighborhood $U$ of $a$. (See \cite[p.~44-46]{Kelley} for more details.) The set of all interior (resp.~boundary) points of $A$ is called the interior (resp.~boundary) of $A$. Note that these concepts are relative in that they depend on the ambient space $X$ containing $A$. For example, if $X=A = [0, +\infty)$, then the interior of $A$ is $A$ itself, and the boundary of $A$ is empty. On the other hand, if $X= \mathbb{R}$ and $A = [0, +\infty)$, then the interior of $A$ is $(0, +\infty)$, and the boundary of $A$ is $\{ 0 \}$.

Now we turn to manifolds. Suppose $M$ is a $d$-dimensional topological manifold (see e.g.,~\cite[p.~231~\&~252]{Hatcher}). By definition, $M$ consists of only two types of points: interior points and boundary points. A point $x \in M$ is called an interior point of $M$ if there is a neighborhood $U$ of $x$ such that $U$ is homeomorphic to $\mathbb{R}^{d}$. On the other hand, $x$ is called a boundary point if there is a neighborhood $U$ of $x$ homeomorphic to the closed half Euclidean space
\[
\mathbb{R}^{d}_{+} = \{ (\xi_{1}, \dots \xi_{d}) \in \mathbb{R}^{d} \mid \xi_{d} \ge 0 \},
\]
and under this homeomorphism, $x$ corresponds to $0 \in \mathbb{R}^{d}_{+}$. By the Invariance of Domain Theorem (\cite[Theorem~2B.3]{Hatcher}), a basic but nontrivial result in algebraic topology, a point $x \in M$ cannot be both an interior point and a boundary point. The subset consisting of all boundary points of $M$, denoted by $\partial M$, is called the boundary of $M$. Its complement $M \setminus \partial M$ is called the interior of $M$. The interior $M \setminus \partial M$ is always nonempty and is a $d$-dimensional manifold without boundary. However, $\partial M$ may be empty. If $\partial M \ne \emptyset$, then it is a $(d-1)$-dimensional manifold without boundary. These notions of interior and boundary in manifold theory differ remarkably from those in point-set topology because these concepts are absolute. More precisely, if we embed $M$ into a larger topological space, then the boundary and interior of $M$ remain unchanged because their definitions are purely intrinsic and do not involve any ambient space.

\emph{In this paper, unless otherwise stated, the terms ``interior" and ``boundary" are understood in the sense of manifold theory. The notation $\partial M$ is always understood as the boundary of $M$ in the sense of manifold theory.}

This convention is convenient for numerical computation. Boundary values are typically assigned to subdomains (see e.g.,~\eqref{alg_continuous_parallel_1}), and in such contexts the boundary is in the sense of manifold theory, not in the point-set topology.

\subsection{Subdomains}\label{subsec_subdomain}
Before addressing subdomains, we recall a method for extending a manifold with boundary to a manifold without boundary.

Suppose $M$ is a $d$-dimensional manifold with boundary $\partial M$. We glue $M$ and $\partial M \times [0,1)$ by identifying $x \in \partial M \subset M$ with $(x,0) \in \partial M \times [0,1)$, i.e.,~take a topological quotient space of the disjoint union of $M$ and $\partial M \times [0,1)$. This yields a $d$-dimensional manifold $\widehat{M}$ containing $M$. This $\widehat{M}$ has no boundary and is noncompact even if $M$ is compact. The original topology of $M$ equals that inherited from $\widehat{M}$. (See Fig.~\ref{fig_collar} for an illustration, in which the shadowed domain on the left is $M$, the larger domain on the right is $\widehat{M}$, the shadowed region in $\widehat{M}$ is $M$, the unshadowed region in $\widehat{M}$ is $\partial M \times (0,1)$, the dashed curve means $\widehat{M}$ has no boundary.)
\begin{figure}[htbp]
\centering
  \includegraphics[width=0.6\textwidth]{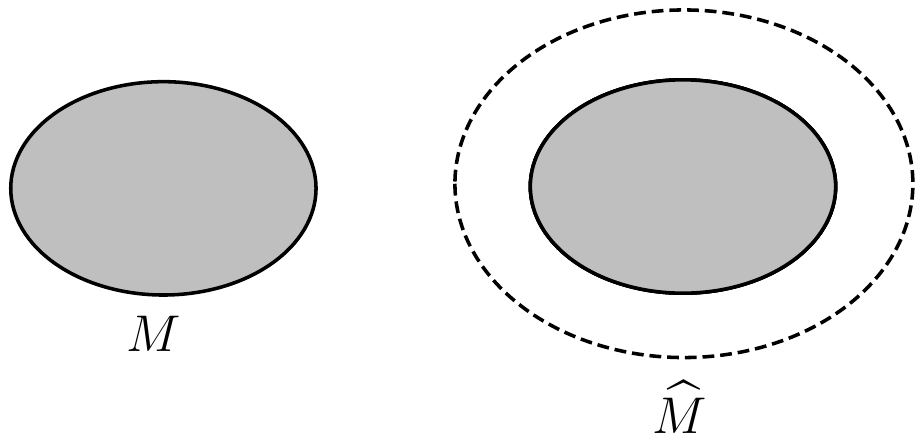}
  \caption{Extend $M$ to $\widehat{M}$.}
  \label{fig_collar}
\end{figure}

We define a subdomain $M_{i}$ of $M$ as a connected, compact, topologically embedded submanifold with boundary and of codimension $0$. More precisely, $M_{i}$ satisfies three conditions. First, $M_{i}$ with the topology inherited from $M$ is a compact topological manifold with boundary and with $\dim M_{i} = \dim M$. Second, $M_{i}$ is connected. (Since $M_{i}$ is a manifold, connectedness is equivalent to path-connectedness.) Third, $M_{i}$ is topologically embedded in the following sense: Suppose $x \in \partial M_{i}$. If $x \notin \partial M$, particularly $\partial M = \emptyset$, then there exists a neighborhood $U$ of $x$ in $M$ such that: (i) $U$ is homeomorphic to $\mathbb{R}^{d}$; (ii) under this homeomorphism, $x$ corresponds to $0$ and $U \cap M_{i}$ corresponds to $\mathbb{R}^{d}_{+}$. (Fig.~\ref{fig_subdomain} illustrates the situation: the large domain on the left is $M$; the subdomain filled with slanted lines is $M_{i}$; the region enclosed by the dashed circle is the neighborhood $U$ of $x$; and $\phi$ is a homeomorphism from $U$ to $\mathbb{R}^{d}$.)
\begin{figure}[htbp]
\centering
  \includegraphics[width=0.6\textwidth]{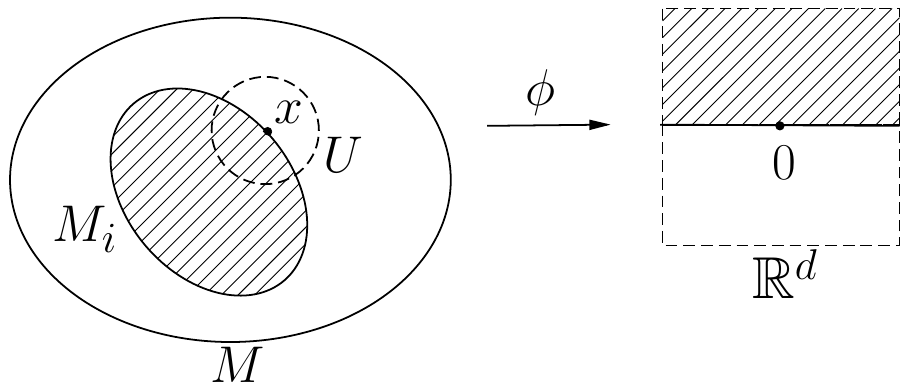}
  \caption{Subdomain $M_{i}$ in $M$.}
  \label{fig_subdomain}
\end{figure}
On the other hand, if $x \in \partial M$, then we require that a neighborhood $U$ of $x$ in $\widehat{M}$ satisfies (i) and (ii) above, where $\widehat{M}$ is the manifold constructed as above. (Fig.~\ref{fig_subdomain_2} illustrates the situation: $M$ is extended to $\widehat{M}$ (see also Fig.~\ref{fig_collar}); the largest domain enclosed by the dashed curve is $\widehat{M}$; the subdomain filled with slanted lines is $M_{i}$; the subdomain intermediate between $M_{i}$ and $\widehat{M}$ is $M$; the symbols $x$, $U$ and $\phi$ are used in the same sense as in Fig.~\ref{fig_subdomain}.)
\begin{figure}[htbp]
\centering
  \includegraphics[width=0.7\textwidth]{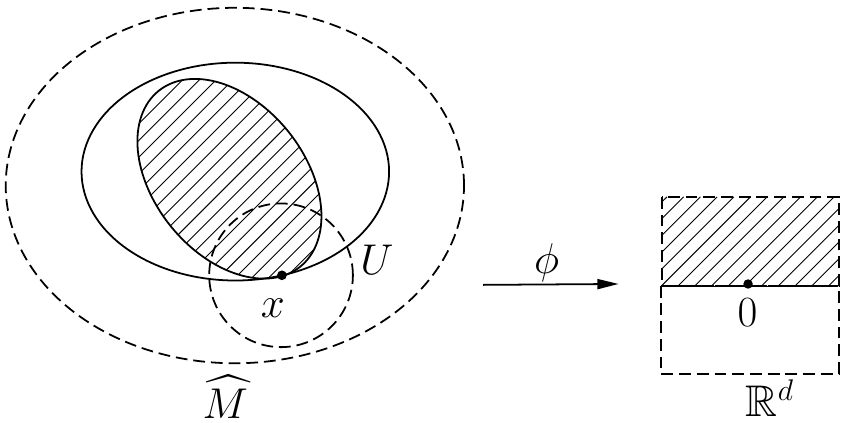}
  \caption{Subdomain $M_{i}$ in $\widehat{M}$.}
  \label{fig_subdomain_2}
\end{figure}

This definition of a subdomain may seem lengthy. However, it already covers a broad family of geometric objects. For example, if $\Omega$ is an open domain in $\mathbb{R}^{d}$, then any $d$-dimensional finite polyhedron in $\Omega$ is a subdomain of $\Omega$ in the above sense. Moreover, the subdomains used in the numerical experiments in \cite{qin_zhang_zhang,cao_qin,qin_wang_wang,jiang_qin_wang} are all subdomains of their respective manifolds in our present sense. On the other hand, this definition is sufficiently restrictive to exclude pathological boundary behaviors, such as the Alexander horned sphere (see \cite[Example~2B.2]{Hatcher}).

By the Invariance of Domain Theorem, $M_{i} \setminus \partial M_{i}$ is an open subset of $M$. Since $M_{i}$ is compact, so is $\partial M_{i}$. Therefore, as subspaces of $M$, the interior (resp.~boundary) of $M_{i}$ in the sense of point-set topology contains (resp.~is contained in) $M_{i} \setminus \partial M_{i}$ (resp.~$\partial M_{i}$). If $M_{i} \cap \partial M = \emptyset$, particularly $\partial M = \emptyset$, then these two types of interiors (resp.~boundary) coincide. On the other hand, if $M_{i} \cap \partial M \ne \emptyset$, then they may differ. Lemma \ref{lem_subdomain_interior} below, frequently used in this paper, addresses this difference.

By Assumption \ref{asp_decomposition}, the manifold $M$ is decomposed into a family of subdomains $M_{i}$ with $1 \le i \le m$, and $\gamma_{i} = \partial M_{i} \setminus \partial M$. Let $\overline{\gamma_{i}}$ denote the closure of $\gamma_{i}$ in $M$. If $M_{i} \cap \partial M = \emptyset$, particularly $\partial M = \emptyset$, then $\gamma_{i} = \overline{\gamma_{i}} = \partial M_{i}$. On the other hand, the situation is very complicated when $M_{i} \cap \partial M \ne \emptyset$. To illustrate this, we first examine some simple examples in dimension $2$.

\begin{example}
Consider the two configurations shown in Fig.~\ref{fig_boundary}, in which the large domain on the left and the large domain on the right are both $M$, with $M_{i} \cap \partial M \ne \emptyset$.

In the left panel, $M_{1}$ is the region bounded by two red curves, one blue curve, and three small circles; $\gamma_{1}$ is the union of the two red curves; the blue curve is $\partial M_{1} \setminus \overline{\gamma_{1}}$; and each small circle represents a point on $\partial M_{i}$ that belongs neither to $\gamma_{1}$ nor to $\partial M_{1} \setminus \overline{\gamma_{1}}$. The red curves and points represented by small circles constitute $\overline{\gamma_{1}}$.

In the right panel, $M_{2}$ is the region bounded by a red curve and a small circle; the red curve is $\gamma_{2}$, the small circle represents $\partial M_{2} \setminus \gamma_{2}$, and $\gamma_{2} \ne \overline{\gamma_{2}} = \partial M_{2}$.

\begin{figure}[htbp]
\centering
  \includegraphics[width=0.7\textwidth]{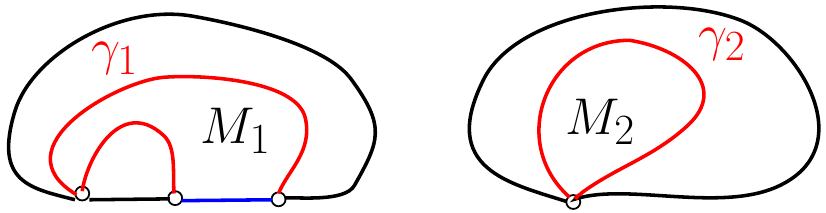}
  \caption{$\gamma_{i}$ and $\partial M_{i}$.}
  \label{fig_boundary}
\end{figure}
\end{example}

In dimension $2$, the general $\gamma_{i}$ and $\overline{\gamma_{i}}$ can be much more complicated than the examples shown in Fig.~\ref{fig_boundary}; let alone the situation in high dimensions. Nevertheless, we can prove the following general and neat conclusion, which matches the intuition conveyed by Fig.~\ref{fig_boundary}.

\begin{lemma}\label{lem_subdomain_interior}
As subsets of $M$, in the sense of point-set topology, $M_{i} \setminus \overline{\gamma_{i}}$ is the interior of $M_{i}$, and $\overline{\gamma_{i}} \setminus \gamma_{i} = \overline{\gamma_{i}} \cap \partial M$.
\end{lemma}
\begin{proof}
Now boundary and interior are in the sense of point-set topology.

Since $\gamma_{i} \subseteq \partial M_{i}$ and $\partial M_{i}$ is closed, we see $\overline{\gamma_{i}} \subseteq \partial M_{i}$. Since $\gamma_{i} \cap \partial M = \emptyset$,
\[
\gamma_{i} \subseteq \overline{\gamma_{i}} \setminus \partial M \subseteq \partial M_{i} \setminus \partial M = \gamma_{i}.
\]
We infer $\overline{\gamma_{i}} \setminus \partial M = \gamma_{i}$ and hence $\overline{\gamma_{i}} \setminus \gamma_{i} = \overline{\gamma_{i}} \cap \partial M$.

The main part of the proof is to verify $M_{i} \setminus \overline{\gamma_{i}}$ is the interior of $M_{i}$. By the Invariance of Domain Theorem (\cite[Theorem~2B.3]{Hatcher}), we know $M_{i} \setminus \partial M_{i}$ lies in the interior of $M_{i}$. Moreover, the topological embedding condition (see Fig.~\ref{fig_subdomain}) implies that $\gamma_{i}$ and hence $\overline{\gamma_{i}}$ are contained in the boundary of $M_{i}$. To finish the proof, it suffices to show, if $x \in \partial M_{i} \setminus \overline{\gamma_{i}}$, then $x$ is in the interior of $M_{i}$.

Before proving this directly, we first claim $M_{i} \setminus \partial M_{i}$ is a connected component of $M \setminus \partial M_{i}$. Clearly, $M \setminus \partial M_{i}$ is a disjoint union of $M_{i} \setminus \partial M_{i}$ and $M \setminus M_{i}$. Since both $M_{i} \setminus \partial M_{i}$ and $M \setminus M_{i}$ are open, they are mutually disconnected. Since $M_{i}$ is connected, so is $M_{i} \setminus \partial M_{i}$. We infer  $M_{i} \setminus \partial M_{i}$ is a connected component of $M \setminus \partial M_{i}$.

Let's study $x \in \partial M_{i} \setminus \overline{\gamma_{i}}$. We necessarily have $x \in \partial M \setminus \overline{\gamma_{i}}$. Since $\overline{\gamma_{i}}$ is closed, and by the definition of $\partial M$, there exists an open neighborhood $B$ of $x$ in $M$ such that the following hold: (1) $B \cap \overline{\gamma_{i}} = \emptyset$; (2) $B$ is homeomorphic to $\mathbb{R}^{d}_{+}$; (3) under this homeomorphism, $B \setminus \partial M$ corresponds to $\mathbb{R}^{d}_{+} \setminus (\mathbb{R}^{d-1} \times \{ 0 \})$.

By the (1), we know
\[
B \setminus \partial M = (B \setminus \gamma_{i}) \setminus \partial M = (B \setminus (\partial M_{i} \setminus \partial M)) \setminus \partial M = B \setminus (\partial M_{i} \cup \partial M) \subseteq M \setminus \partial M_{i}.
\]
Since $x \in \partial M_{i}$, it is in the closure of $M_{i} \setminus \partial M_{i}$, so there is a $y \in B \cap (M_{i} \setminus \partial M_{i})$. Clearly, $y \notin \partial M$, we have $y \in (B \setminus \partial M) \cap (M_{i} \setminus \partial M_{i})$ and hence
\[
(B \setminus \partial M) \cap (M_{i} \setminus \partial M_{i}) \ne \emptyset.
\]
Since $B \setminus \partial M$ is connected and $M_{i} \setminus \partial M_{i}$ is a connected component of $M \setminus \partial M_{i}$, we infer $(B \setminus \partial M) \subseteq (M_{i} \setminus \partial M_{i})$ and hence $B \subseteq \overline{B \setminus \partial M} \subseteq M_{i}$. Since $B$ is a neighborhood of $x$, we see $x$ is an interior point of $M_{i}$.
\end{proof}

The following lemma is needed for the proof of Proposition \ref{prop_winding}.
\begin{lemma}\label{lem_boundary_interior}
For each $i$, in the sense of point-set topology, $\overline{\gamma_{i}} \cap \partial M$ has no relative interior in $\partial M$.
\end{lemma}
\begin{proof}
We prove by contradiction. Suppose $x \in \overline{\gamma_{i}} \cap \partial M$ is a relative interior point of $\overline{\gamma_{i}} \cap \partial M$ in $\partial M$. Then there would be an $A \subseteq \partial M$ such that $A$ is relatively open in $\partial M$ and $x \in A \subseteq \overline{\gamma_{i}} \cap \partial M$.

Note that $A \subseteq \overline{\gamma_{i}} \subseteq \partial M_{i}$. Since both $\partial M$ and $\partial M_{i}$ are $(d-1)$-dimensional topological manifolds, by the Invariance of Domain Theorem (\cite[Theorem~2B.3]{Hatcher}), we infer $A$ would be also relatively open in $\partial M_{i}$. As a result, there would be an open neighborhood $U$ of $x$ in $M$ such that $A = U \cap \partial M_{i}$. Now we would have
\[
U \cap \gamma_{i} = U \cap (\partial M_{i} \setminus \partial M) = (U \cap \partial M_{i}) \setminus \partial M = A \setminus \partial M = \emptyset,
\]
which contradicts the fact that $x \in \overline{\gamma_{i}}$.
\end{proof}

\subsection{Lipschitz Boundary}\label{subsec_lipschitz}
In Assumption \ref{asp_decomposition}, the $\partial M_{i}$ is required to be Lipschitz. We address the Lipschitz property now.

Our target manifold $M$ is smooth, not merely topological. In other words, $M$ is covered by a family of $C^{\infty}$-compatible coordinate charts, where compatibility means transition functions between charts are $C^{\infty}$-diffeomorphisms (see \cite{Warner}). As described above, we can extend a manifold $M$ with boundary $\partial M$ to manifold $\widehat{M}$ without boundary by gluing $M$ and $\partial M \times [0,1)$ along $\partial M$. When $M$ is smooth, the resulting $\widehat{M}$ also admits a smooth structure compatible with those of $M$ and $\partial M \times [0,1)$; moreover, this smooth structure on $\widehat{M}$ is unique up to $C^{\infty}$-diffeomorphisms (see Lemmas A.1 and A.2 in \cite[p.~86]{milnor3}). With this construction, we may assume $M$ has no boundary in this subsection.

From the viewpoint of numerical computation, it's too stringent to require a subdomain $M_{i}$ to be a $C^{1}$-embedded submanifold with boundary of $M$, let alone $C^{\infty}$-embedded. For example, a $d$-dimensional finite polyhedron $D$ is not a $C^{1}$-submanifold of $\mathbb{R}^{d}$ because $\partial D$ has corners. This $D$ has merely a Lipschitz boundary.

Let's define the Lipschitz property of a boundary in the setting of manifolds. We say $\partial M_{i}$ is Lipschitz, if $\forall x \in \partial M_{i}$, there is a coordinate chart $U$ of $M$ containing $x$ such that, under this specific coordinate system, $U \cap \partial M_{i}$ is Lipschitz in the usual sense in $\mathbb{R}^{d}$ (see \cite[4.9]{Adams_Fournier}). More precisely, after shrinking $U$ and applying an affine isomorphism of $\mathbb{R}^{d}$ if necessary, there is a $C^{\infty}$-diffeomorphism $\phi \colon U \rightarrow \phi (U) \subset \mathbb{R}^{d}$ satisfying the following two conditions: First, $\phi (x) =0$ in $\mathbb{R}^{d}$ and $\phi (U) = W \times (-\delta, \delta)$, where $W$ is an open neighborhood of $0 \in \mathbb{R}^{d-1}$. Second, there exists a Lipschitz function $l \colon W \rightarrow (-\delta, \delta)$ such that
\[
\phi (U \cap \partial M_{i}) = \mathrm{graph} (l) := \{ (w, l(w)) \in W \times (-\delta, \delta) \mid w \in W \}.
\]
The condition $\phi (x) =0$ implies $l(0) =0$.

Note that a Lipschitz boundary is more restrictive than the boundary of a general subdomain defined in $\S$\ref{subsec_subdomain}. Although the topological lemmas proved in $\S$\ref{subsec_subdomain} do not require the Lipschitz boundary assumption, this extra regularity is important for analysis, see e.g.,~Lemma \ref{lem_perron_lipschitz} below. In practice, one can also easily construct subdomains with Lipschitz boundary, for instance, the subdomains used in the numerical experiments in \cite{qin_zhang_zhang,cao_qin,qin_wang_wang,jiang_qin_wang} all have Lipschitz boundaries.

We emphasize that the above definition of Lipschitz property is independent of the choice of charts. Equivalently, Lipschitz property is preserved under $C^{\infty}$-diffeomorphisms. We can actually prove a stronger statement: this properties is even preserved under $C^{1}$-diffeomorphisms. The crux of the matter is reduced to the following problem. Suppose $l \colon B^{d-1}_{r} \rightarrow (-\delta, \delta)$ is a Lipschitz function with $l(0) =0$, where $B^{d-1}_{r} := \{ w \in \mathbb{R}^{d-1} \mid \| w \| <r \}$. Suppose further
\[
H \colon B^{d-1}_{r} \times (-\delta, \delta) \rightarrow H(B^{d-1}_{r} \times (-\delta, \delta)) \subseteq \mathbb{R}^{d}
\]
is a $C^{1}$-diffeomorphism with $H(0) =0$ and $\mathrm{D} H(0) =I$, where $\mathrm{D} H$ is the differential of $H$. We need to check that $H (\mathrm{graph} (l)) \cap V$ is the graph of a Lipschitz function for some open neighborhood $V$ of $0 \in \mathbb{R}^{d}$. The verification proceeds as follows.

Define $R(x) := H(x) -x$. Then, $\forall x,y \in B^{d-1}_{r} \times (-\delta, \delta)$,
\begin{align*}
R(x) - R(y) & = \int_{0}^{1} \frac{\mathrm{d}}{\mathrm{d} t} H(y+t(x-y)) \mathrm{d} t - (x-y) \\
& = \int_{0}^{1} [\mathrm{D} H (y+t(x-y)) - I] \mathrm{d} t \cdot (x-y).
\end{align*}
Since $\mathrm{D} H$ is continuous and $\mathrm{D} H(0) =I$, we infer, $\forall \epsilon >0$, $\exists r' \in (0,r)$, $\exists \delta' \in (0, \delta)$, such that $\forall x,y \in B^{d-1}_{r'} \times (-\delta', \delta')$,
\begin{equation}\label{eqn_remainder_lipschitz}
\| R (x) - R(y) \| \le \epsilon \| x - y \|.
\end{equation}

Write $H(x) = (H_{1} (x), H_{2} (x))$, where $H_{1} (x) \in \mathbb{R}^{d-1}$ and $H_{2} (x) \in \mathbb{R}$. Similarly, write $R(x) = (R_{1} (x), R_{2} (x))$. Define a map $\Psi \colon B^{d-1}_{r} \rightarrow \mathbb{R}^{d-1}$ as $\Psi (w) = H_{1} (w, l(w))$. Then $\forall w_{1}, w_{2} \in B^{d-1}_{r}$, we have
\begin{align*}
\| \Psi (w_{1}) - \Psi (w_{2}) \| & = \| w_{1} - w_{2} + (R_{1} (w_{1}, l(w_{1}))- R_{1} (w_{2}, l(w_{2}))) \| \\
& \ge \| w_{1} - w_{2} \| - \| R_{1} (w_{1}, l(w_{1}))- R_{1} (w_{2}, l(w_{2})) \| \\
& \ge \| w_{1} - w_{2} \| - \| R (w_{1}, l(w_{1}))- R (w_{2}, l(w_{2})) \|.
\end{align*}
Since $l$ is Lipschitz, by \eqref{eqn_remainder_lipschitz}, we infer, $\exists r' \in (0,r)$, such that $\forall w_{1}, w_{2} \in B^{d-1}_{r'}$,
\[
\| R (w_{1}, l(w_{1}))- R (w_{2}, l(w_{2})) \| \le \frac{1}{2} \| w_{1} - w_{2} \|
\]
and hence
\[
\| \Psi (w_{1}) - \Psi (w_{2}) \| \ge \frac{1}{2} \| w_{1} - w_{2} \|.
\]
Thus $\Psi^{-1}$ exists and is Lipschitz on $\Psi (B^{d-1}_{r'})$. We also know $\Psi (B^{d-1}_{r'})$ is an neighborhood of $0 \in \mathbb{R}^{d-1}$. This topological fact immediately follows from Invariance of Domain Theorem (\cite[Theorem~2B.3]{Hatcher}), but it can also be proved by the more elementary Banach Fixed Point Theorem.

Define $l' \colon \Psi (B^{d-1}_{r'}) \rightarrow \mathbb{R}$ as the composition
\[
l' \colon \Psi (B^{d-1}_{r'})  \overset{\Psi^{-1}}{\longrightarrow} B^{d-1}_{r'} \overset{w \mapsto (w,l(w))}{\longrightarrow} B^{d-1}_{r'} \times (-\delta, \delta) \overset{H_{2}}{\longrightarrow} \mathbb{R}.
\]
Then $l'$ is Lipschitz and $\mathrm{graph} (l') \subseteq H (\mathrm{graph} (l))$. Since $H$ is a homeomorphism, $H (B^{d-1}_{r'} \times (-\delta, \delta))$ is open in $\mathbb{R}^{d}$. Choose $r''>0$ such that $\Psi (B^{d-1}_{r'}) \supseteq B^{d-1}_{r''}$. Define
\[
V := (B^{d-1}_{r''} \times \mathbb{R}) \cap H (B^{d-1}_{r'} \times (-\delta, \delta)).
\]
Then $V$ is an open neighborhood of $0 \in \mathbb{R}^{d}$ and $V \cap H (\mathrm{graph} (l))$ is the graph of $l'$.

We have thus proved that the Lipschitz property of $\partial M_{i}$ is independent of the choice of charts, provided that the atlas is $C^{1}$-compatible.

It is worth pointing out, however, that even if a subdomain takes the form of a polyhedron in one particular chart, it need not do so in another, since the transition functions between charts are generally nonlinear.

\subsection{Winding Number}\label{subsec_winding}
Now let's study the concept of winding number in Definition \ref{def_winding}. We assume $\partial M \ne \emptyset$ throughout this subsection.

To gain intuition, we first examine some simple examples in dimension $2$.
\begin{example}
In the left panel of Fig.~\ref{fig_winding}, the domain enclosed by the outer red curve is $M$. The annulus bounded by the two red curves is $M_{1}$, the region bounded by the blue curve is $M_{2}$. This gives a decomposition $\mathcal{D} := \{ M_{1}, M_{2} \}$ of $M$. Suppose $\rho := \{ \rho_{1}, \rho_{2} \}$ is a partition of unity subordinate to $\mathcal{D}$. Then $\partial M_{1} = \partial M \cup \gamma_{1}$, where $\gamma_{1}$ is the inner red curve, and $M_{2} \cap \partial M = \emptyset$. Thus $I_{1} (\mathcal{D}, \rho) = \{ 1 \}$. Since $\rho_{2}|_{\partial M_{2}} =0$, we infer $\rho_{1}|_{\partial M_{2}} =1$ and $\partial M_{2} \subseteq U_{1} (\rho)$. So $I_{2} (\mathcal{D}, \rho) = \{ 1,2 \}$ and the winding number $N(\mathcal{D}, \rho) =2$. (Recall that $I_{i}$ and $U_{i}$ are defined prior to Definition \ref{def_winding}.)

In the right panel of Fig.~\ref{fig_winding}, we add to $\mathcal{D}$ one more subdomain $M_{3}$ which is bounded by the green curve. Let $\widetilde{\mathcal{D}} := \{ M_{1}, M_{2}, M_{3} \}$. Suppose $\tilde{\rho} := \{ \tilde{\rho}_{1}, \tilde{\rho}_{2}, \tilde{\rho}_{3} \}$ is subordinate to $\widetilde{\mathcal{D}}$. Since $\partial M_{3} = \overline{\gamma_{3}}$, we still have $I_{1} (\widetilde{\mathcal{D}}, \tilde{\rho}) = \{ 1 \}$. On the other hand, since $\partial M_{3} \cap \partial M = \{ x \}$, we have $\tilde{\rho}_{2} (x) = \tilde{\rho}_{3} (x) =0$ and hence $\tilde{\rho}_{1} (x) =1$. Consequently, $x \in \partial M_{3} \cap U_{1} (\tilde{\rho})$ and $3 \in I_{2} (\widetilde{\mathcal{D}}, \tilde{\rho})$. (Alternatively, since $\partial M_{3} \cap \partial M \ne \emptyset$, the general Lemma \ref{lem_boundary_neighborhood} below also implies $3 \in I_{2} (\widetilde{\mathcal{D}}, \tilde{\rho})$.) Note that, $y \in \partial M_{2}$ and $\tilde{\rho}_{2} (y) = \tilde{\rho}_{3} (y) =0$. So $y \in \partial M_{2} \cap U_{1} (\tilde{\rho})$ and $2 \in I_{2} (\widetilde{\mathcal{D}}, \tilde{\rho})$. Hence $I_{2} (\widetilde{\mathcal{D}}, \tilde{\rho}) = \{ 1,2,3 \}$, and therefore $N (\widetilde{\mathcal{D}}, \tilde{\rho}) =2$.
\begin{figure}[htbp]
\centering
  \includegraphics[width=0.7\textwidth]{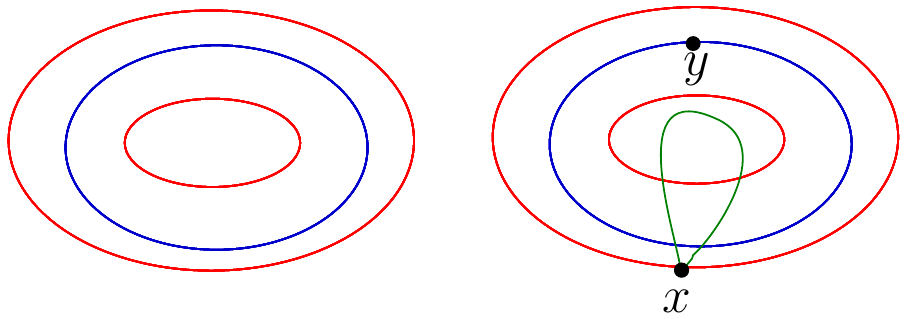}
  \caption{Winding number.}
  \label{fig_winding}
\end{figure}
\end{example}

We proceed to prove several general results on winding numbers.
\begin{lemma}\label{lem_winding_start}
$I_{1} \ne \emptyset$.
\end{lemma}
\begin{proof}
We know $\overline{\gamma_{i}} \cap \partial M$ is closed for each $i$. By Lemma \ref{lem_boundary_interior}, $\bigcup_{i=1}^{m} (\overline{\gamma_{i}} \cap \partial M)$ is also closed and has no relative interior in $\partial M$. Thus
\[
\partial M \cap \bigcup_{i=1}^{m} \overline{\gamma_{i}} = \bigcup_{i=1}^{m} (\overline{\gamma_{i}} \cap \partial M) \ne \partial M.
\]
However, $\partial M \subseteq \bigcup_{i=1}^{m} M_{i}$ and $(M_{i} \setminus \partial M_{i}) \cap \partial M = \emptyset$ for all $i$. So $\partial M \subseteq \bigcup_{i=1}^{m} \partial M_{i}$. We infer $\partial M_{j} \ne \overline{\gamma_{j}}$ for some $j$ and hence $I_{1} \ne \emptyset$.
\end{proof}

\begin{lemma}\label{lem_winding}
Suppose $n > 1$, the following hold:
\begin{enumerate}[(1)]
\item $i \in I_{n}$ if and only if $\partial M_{i} \ne \overline{\gamma_{i}}$ or $\partial M_{i} \cap (\bigcup_{j \in I_{n-1}} U_{j}) \ne \emptyset$.

\item Assume further $n \le N$, then $i \in I_{n} \setminus I_{n-1}$ if and only if $\partial M_{i} = \overline{\gamma_{i}}$, $\partial M_{i} \cap ( \bigcup_{j \in I_{n-2}} U_{j} ) = \emptyset$ and $\partial M_{i} \cap ( \bigcup_{j \in I_{n-1} \setminus I_{n-2}} U_{j} ) \ne \emptyset$.
\end{enumerate}
\end{lemma}
\begin{proof}
The (1) follows by induction on $n$. The (2) follows from (1).
\end{proof}

It's ready to prove Proposition \ref{prop_winding}.
\begin{proof}[Proof of Proposition \ref{prop_winding}]
We have to show $i_{0} \in I_{N}$ if $1 \le i_{0} \le m$. Fix the $i_{0}$. In this proof, boundary and interior are in the sense of point-set topology.

Define $J_{1} = \{ i_{0} \}$. For $n>1$, define inductively that $D_{n-1} = \bigcup_{k=1}^{n-1} \bigcup_{j \in J_{k}} M_{j}$ and
\[
J_{n} = \{ j \mid \text{$U_{j}$ contains a boundary point of $D_{n-1}$ in $M$} \}.
\]
Since $M_{i} \setminus \partial M_{i}$ is open for each $i$, a boundary point of $D_{n}$ is contained in $\partial M_{j}$ for some $j \in \bigcup_{k=1}^{n} J_{k}$.

Clearly, $U_{j} \subset M_{j}$. Since $U_{j}$ is open, we further infer $U_{j}$ does not contain a boundary point of $D_{n-1}$ for each $j \in \bigcup_{k=1}^{n-1} J_{k}$. So these $J_{n}$'s mutually have no intersection. This implies that there exists an $n_{0}$ such that $J_{n_{0}} \ne \emptyset$ and $J_{n} = \emptyset$ for $n> n_{0}$.

We claim that $D_{n_{0}} = M$. Actually, $D_{n_{0}}$ has no boundary. Otherwise, since $\bigcup_{j=1}^{m} U_{j} = M$, we would have $J_{n_{0} +1} \ne \emptyset$, which is a contradiction. Clearly, $D_{n_{0}}$ is closed. It is also open because it has no boundary. Since $M$ is connected, $D_{n_{0}} = M$.

By Lemma \ref{lem_boundary_interior}, $\overline{\gamma_{j}} \cap \partial M$ has no relative interior in $\partial M$ for each $j$. On the other hand, $D_{n_{0}} = M$ and hence
\[
\partial M = \bigcup_{k=1}^{n_{0}} \bigcup_{j \in J_{k}} (\partial M_{j} \cap \partial M ).
\]
Since each $\partial M_{j} \cap \partial M$ is closed, we infer there exists an $n_{1} \le n_{0}$ such that $\partial M_{j_{1}} \cap \partial M$ has relative interior in $\partial M$ for some $j_{1} \in J_{n_{1}}$. Clearly, $\partial M_{j_{1}} \ne \overline{\gamma_{j_{1}}}$ and $j_{1} \in I_{1}$. By the definition of $J_{n_{1}}$, we know $U_{j_{1}}$ contains a boundary point $x_{0}$ of $D_{n_{1} -1}$. This $x_{0}$ must be a boundary point of $M_{j_{2}}$ for some $j_{2} \in J_{n_{2}}$ with $n_{2} < n_{1}$. Now $x_{0} \in \partial M_{j_{2}} \cap U_{j_{1}}$ and hence $\partial M_{j_{2}} \cap U_{j_{1}} \ne \emptyset$, we have $j_{2} \in I_{2}$. By an inductive argument, we get a decreasing sequence $n_{1} > n_{2} > \dots > n_{s} =1$ and indices $j_{1} \in J_{n_{1}} \cap I_{1}$, $j_{2} \in J_{n_{2}} \cap I_{2}$, \dots, $j_{s} \in J_{n_{s}} \cap I_{s}$. Here $J_{n_{s}} = J_{1}$ and, therefore, $j_{s} = i_{0}$. We eventually see $i_{0} \in I_{s} \subseteq I_{N}$, which completes the proof. (Note that it might be that $s >N$.)
\end{proof}

\begin{lemma}\label{lem_boundary_neighborhood}
The following hold:
\begin{enumerate}[(1)]
\item $\bigcup_{i \in I_{1}} U_{i}$ is an open neighborhood of $\partial M$.

\item If $M_{i} \cap \partial M \ne \emptyset$, then $i \in I_{2}$.
\end{enumerate}
\end{lemma}
\begin{proof}
(1). Since $\bigcup_{i \in I_{1}} U_{i}$ is open, it suffices to show $\partial M \subseteq \bigcup_{i \in I_{1}} U_{i}$. Note that $U_{i} \subset M_{i} \setminus \overline{\gamma_{i}}$. If $U_{i} \cap \partial M \ne \emptyset$, then $(M_{i} \setminus \overline{\gamma_{i}}) \cap \partial M \ne \emptyset$ and hence $i \in I_{1}$. Now the conclusion follows from the fact $\partial M \subseteq M = \bigcup_{i=1}^{m} U_{i}$.

(2). Since $\partial M_{i} \cap \partial M = M_{i} \cap \partial M$, this follows immediately from (1).
\end{proof}

\section{Basic Theory of Elliptic Equations}\label{sec_basic}
In this section, we recall some basic results on second-order linear elliptic equations on manifolds. For domains in  $\mathbb{R}^{d}$, these results are standard and can be readily found in many textbooks (see e.g.,~\cite{Gilbarg_Trudinger}). The extension to manifolds is also well-known, though explicit references are less commonly available in the literature. For the reader's convenience, we include a brief derivation of the manifold version from the Euclidean one.

\subsection{Maximum Principle}
The maximum principle plays a central role in the present work.

Let $\mathfrak{L}$ be the one in \eqref{eqn_operator}. Let $M'$ be a subdomain of $M$ (see $\S$\ref{subsec_subdomain} for the definition of a subdomain). Here $\partial M'$ is allowed to be empty, in that case, we necessarily have $M' = M$.
\begin{theorem}[Maximum Principle]\label{thm_maximum}
Suppose $v \in C^{2} (M' \setminus \partial M')$ and $\mathfrak{L} v \le 0$ (resp.~$\ge 0$). Then the following hold:
\begin{enumerate}[(1)]
\item (Strong Maximum Principle) The $v$ cannot attain its nonnegative maximum (resp. nonpositive minimum) in $M' \setminus \partial M'$ unless it is a constant.

\item (Weak Maximum Principle) If $\partial M' \ne \emptyset$ and $v \in C^{0} (M')$, then $\max\limits_{M'} v \le \max\limits_{\partial M'} v^{+}$ (resp.~$\min\limits_{M'} v \ge \min\limits_{\partial M'} v^{-}$), where $v^{+} := \max \{ v, 0 \}$ and $v^{-} := \min \{ v, 0 \}$.
\end{enumerate}
\end{theorem}
\begin{proof}
(1). Since $\mathfrak{L} v \le 0$ if and only if $\mathfrak{L} (-v) \ge 0$, it suffices to prove the case of $\mathfrak{L} v \le 0$.

Suppose $v$ attains its maximum $\eta \ge 0$. Let
\[
S := \{ p \in M' \setminus \partial M' \mid v(p) = \eta \}.
\]
Clearly, $S$ is relatively closed in $M' \setminus \partial M'$. Suppose $p_{0} \in S$. There is an open and connected neighborhood $\Omega$ of $p_{0}$ such that $\overline{\Omega} \subset \subset M' \setminus \partial M'$ and $\overline{\Omega}$ has coordinates. Now $\Omega$ is identified with an open domain of $\mathbb{R}^{d}$, and $\mathfrak{L}$ can be expressed as \eqref{eqn_operator_local}, which is a uniformly elliptic equation with $C^{\infty}$ coefficients. Since $c \ge 0$ and $\Omega$ is connected, by the strong maximum principle on Euclidean domains (\cite[Theorem~3.5]{Gilbarg_Trudinger}), we see $v \equiv \eta$ in $\Omega$. Thus $\Omega \subseteq S$ and $S$ is also open. Since $M' \setminus \partial M'$ is connected and $S$ is nonempty, we infer $M' \setminus \partial M' = S$ and the conclusion follows.

(2). By the assumption, $v$ attains its maximum and minimum on $M'$. The conclusion follows from (1).
\end{proof}

\begin{corollary}[Maximum Principle]\label{cor_maximum}
Suppose $v \in C^{2} (M' \setminus \partial M')$ and $\mathfrak{L} v = 0$. Then the following hold:
\begin{enumerate}[(1)]
\item (Strong Maximum Principle) The $v$ can neither attain its nonnegative maximum, nor attain its nonpositive minimum, in $M' \setminus \partial M'$ unless it is a constant.

\item (Weak Maximum Principle) If $\partial M' \ne \emptyset$ and $v \in C^{0} (M')$, then $\max\limits_{M'} v \le \max\limits_{\partial M'} v^{+}$, $\min\limits_{M'} v \ge \min\limits_{\partial M'} v^{-}$, and $\max\limits_{M'} |v| = \max\limits_{\partial M'} |v|$.
\end{enumerate}
\end{corollary}

\begin{corollary}[Comparison Principle]\label{cor_comparison}
Suppose $v_{i} \in C^{2} (M' \setminus \partial M') \cap C^{0} (M')$ for $i=1,2$ and $\partial M' \ne \emptyset$. If $\mathfrak{L} v_{1} \le \mathfrak{L} v_{2}$ and $v_{1}|_{\partial M'} \le v_{2}|_{\partial M'}$, then $v_{1} \le v_{2}$. If $v_{1}|_{\partial M'} \ne v_{2}|_{\partial M'}$ additionally, then $v_{1} < v_{2}$ in $M' \setminus \partial M'$.
\end{corollary}
\begin{proof}
Apply the weak maximum principle \ref{thm_maximum} to $v_{1} - v_{2}$, we see $v_{1} \le v_{2}$. If additionally $v_{1}|_{\partial M'} \ne v_{2}|_{\partial M'}$, then $v_{1} - v_{2}$ cannot vanish in $M' \setminus \partial M'$. We obtain $v_{1} < v_{2}$ by the strong maximum principle.
\end{proof}

\subsection{Perron Solution}\label{subsec_Perron}
Let's recall the concept of Perron solution. Now let $M'$ be a subdomain of $M$ with boundary $\partial M'$. Let $\psi$ be a bounded function on $\partial M'$. Let $S_{\psi}$ be the set consisting of subfunctions with respect to $\psi$. Here by a subfunction $v$, we mean $v \in C^{0} (M')$ with $v|_{\partial M'} \le \psi$, and $v|_{M' \setminus \partial M'}$ is a subsolution to $\mathfrak{L} w=0$ (see e.g.,~\cite[p.~103]{Gilbarg_Trudinger}, \cite[Definition~2.2]{crandall_ishii_lions} and \cite[Definition~2.6]{qin_wang_wang}). Define $u_{\psi} := \sup_{v \in S_{\psi}} v$. We call $u_{\psi}$ the Perron solution to $\mathfrak{L} =0$ with boundary value $\psi$.

\begin{lemma}\label{lem_perron}
The following hold:
\begin{enumerate}[(1)]
\item The $u_{\psi}$ is a well-defined bounded function on $M'$. In $M' \setminus \partial M'$, we have $u_{\psi} \in C^{\infty}$ and $\mathfrak{L} u_{\psi} =0$.

\item If $\lambda_{1} \le 0$ and $\lambda_{2} \ge 0$ are constants such that $\lambda_{1} \le \psi$ (resp.~$\lambda_{2} \ge \psi$), then $\lambda_{1} \le u_{\psi}$ (resp.~$\lambda_{2} \ge u_{\psi}$), and $u_{\psi}$ cannot achieve $\lambda_{1}$ (resp.~$\lambda_{2}$) in $M' \setminus \partial M'$ unless it is a constant.

\item If $\lambda \ge 0$ is a constant, then $\lambda u_{\psi}$ is the Perron solution with boundary value $\lambda \psi$.

\item Suppose $\psi_{1} \le \psi_{2}$ on $\partial M'$. Then $u_{\psi_{1}} \le u_{\psi_{2}}$. Furthermore, if $u_{\psi_{1}}$ is not identical to $u_{\psi_{2}}$ in $M' \setminus \partial M'$, then $u_{\psi_{1}}|_{M' \setminus \partial M'} < u_{\psi_{2}}|_{M' \setminus \partial M'}$.
\end{enumerate}
\end{lemma}
\begin{proof}
(1). First of all, $S_{\psi} \ne \emptyset$ because $\psi$ is bounded. For instance, let $\mu <0$ be a constant such that $\mu < \inf \psi$. Then $v \equiv \mu$ belongs to $S_{\psi}$. The remaining argument is essentially local and can be reduced to the case of problems in Euclidean domains by taking local charts. It follows those of \cite[Theorems~6.11~\&~6.17]{Gilbarg_Trudinger}.

(2). By the assumption, $v \equiv \lambda_{1}$ is in $S_{\psi}$. Thus $\lambda_{1} \le u_{\psi}$. On the other hand, $v_{0} \equiv \lambda_{2}$ is a supersolution with $v_{0}|_{\partial M'} \ge \psi$. So $v \le \lambda_{2}$ for all $v \in S_{\psi}$. We infer $u_{\psi} \le \lambda_{2}$. The remaining statement follows from the strong maximum principle \ref{cor_maximum}.

(3). If $\lambda =0$, the conclusion is obviously true. If $\lambda >0$, then $S_{\lambda \psi} = \{ \lambda v \mid v \in  S_{\psi} \}$, the conclusion also follows.

(4). Since $\psi_{1} \le \psi_{2}$, we have $S_{\psi_{1}} \subseteq S_{\psi_{2}}$. Then $u_{\psi_{1}} \le u_{\psi_{2}}$. Since $\mathfrak{L} (u_{\psi_{1}} - u_{\psi_{2}}) =0$ in $M' \setminus \partial M'$, the last statement follows from the strong maximum principle \ref{cor_maximum}.
\end{proof}

\begin{lemma}\label{lem_perron_lipschitz}
Assume further $\partial M'$ is Lipschitz (see $\S$\ref{subsec_lipschitz}), the following hold:
\begin{enumerate}[(1)]
\item If $\psi$ is continuous at $x \in \partial M'$, then $u_{\psi}$ is continuous at $x$ and $u_{\psi} (x) = \psi (x)$.

\item If $\psi$ is continuous on $\partial M'$, then $u_{\psi}$ is the unique one in $C^{0} (M') \cap C^{2} (M' \setminus \partial M')$ such that $u_{\psi}|_{\partial M'} = \psi$ and $\mathfrak{L} u_{\psi} =0$ in $M' \setminus \partial M'$.
\end{enumerate}
\end{lemma}
\begin{proof}
(1). It suffices to construct a barrier near $x$. The argument is also local, we may assume $M'$ is a domain in $\mathbb{R}^{d}$. Since $\partial M'$ is Lipschitz, it satisfies the exterior cone condition. By \cite[Problem~6.3]{Gilbarg_Trudinger}, the conclusion holds.

(2). The conclusion follows from (1) of this lemma, (1) of Lemma \ref{lem_perron} and the weak maximum principle \ref{cor_maximum}.
\end{proof}

\section{Convergence Rate}\label{sec_rate}
In this section, we prove Theorems \ref{thm_bound}, \ref{thm_comparison}, and \ref{thm_refine}.

\begin{lemma}\label{lem_bound_perron}
The $\theta_{i}$ in \eqref{eqn_bound_function} satisfies the following:
\begin{enumerate}[(1)]
\item $\theta_{i}|_{\partial M_{i} \setminus \overline{\gamma_{i}}} =0$ and $\theta_{i}|_{\gamma_{i}} =1$, and $\theta_{i}$ is continuous on
\[
(M_{i} \setminus \partial M_{i}) \cup \gamma_{i} \cup ({\partial M_{i} \setminus \overline{\gamma_{i}}}) = (M_{i} \setminus \overline{\gamma_{i}}) \cup \gamma_{i} = M_{i} \setminus (\overline{\gamma_{i}} \cap \partial M);
\]

\item $0 \le \theta_{i} \le 1$;

\item $\theta_{i}|_{M_{i} \setminus \partial M_{i}}>0$, and if $\theta_{i}|_{M_{i} \setminus \partial M_{i}}$ attains $1$, then $\theta_{i}|_{M_{i} \setminus \partial M_{i}} =1$;

\item if $\partial M_{i} \ne \overline{\gamma_{i}}$, then $\theta_{i} < 1$ in $M_{i} \setminus \overline{\gamma_{i}}$;

\item if in \eqref{eqn_operator}, $c>0$ somewhere in $M_{i}$, then $\theta_{i} < 1$ in $M_{i} \setminus \partial M_{i}$.
\end{enumerate}
\end{lemma}
\begin{proof}
(1). Both $\gamma_{i}$ and $\partial M_{i} \setminus \overline{\gamma_{i}}$ are relatively open in $\partial M_{i}$. So the prescribed boundary values in \eqref{eqn_bound_function} are continuous at each point in $\gamma_{i} \cup (\partial M_{i} \setminus \overline{\gamma_{i}})$. By Lemma \ref{lem_subdomain_interior}, we have $\overline{\gamma_{i}} \setminus \gamma_{i} = \overline{\gamma_{i}} \cap \partial M$ and hence
\[
\gamma_{i} \cup (\partial M_{i} \setminus \overline{\gamma_{i}}) = \partial M_{i} \setminus (\overline{\gamma_{i}} \setminus \gamma_{i}) = \partial M_{i} \setminus (\overline{\gamma_{i}} \cap \partial M).
\]
Since $\partial M_{i}$ is Lipschitz, the conclusion follows from Lemma \ref{lem_perron_lipschitz}.

(2) and (3). By the assumption, we know $\gamma_{i} \ne \emptyset$, otherwise we would have $M_{i} = M$. By (1), we have $\theta_{i}|_{\gamma_{i}} =1$ and $\theta_{i}$ is continuous at each point of $\gamma_{i}$. The conclusions follow from (2) of Lemma \ref{lem_perron}.

(4). By (1), we have $\theta_{i}|_{\partial M_{i} \setminus \overline{\gamma_{i}}} =0$ and $\theta_{i}$ is continuous at each point in $\partial M_{i} \setminus \overline{\gamma_{i}}$. The conclusion follows from (3).

(5). By the continuity of $c$, we have $c(x) >0$ for some $x \in M_{i} \setminus \partial M_{i}$. If $\theta_{i}$ attains $1$ in $M_{i} \setminus \partial M_{i}$, by (3), we would have $\theta_{i}|_{M_{i} \setminus \partial M_{i}} =1$. Then $\mathfrak{L} \theta_{i} (x) >0$, which is a contradiction.
\end{proof}

\begin{proof}[Proof of Theorem \ref{thm_bound}]
By (2) and (3) of Lemma \ref{lem_bound_perron}, we infer $0 < \Theta_{i} \le 1$ for each $i$, hence $0 < \Theta \le 1$.

If $\partial M_{i} \ne \overline{\gamma_{i}}$, by (4) of Lemma \ref{lem_bound_perron}, we infer $\theta_{i}|_{M_{i} \setminus \overline{\gamma_{i}}} <1$. Since $\mathrm{supp} \rho_{i}$ is compact, $\mathrm{supp} \rho_{i} \subset M_{i} \setminus \overline{\gamma_{i}}$ and $\theta_{i}|_{\mathrm{supp} \rho_{i}}$ is continuous, we infer $\Theta_{i} <1$. Suppose $c>0$ somewhere in $M_{i}$. By (5) of Lemma \ref{lem_bound_perron}, similarly, we obtain $\Theta_{i} <1$.

In summary, if $c>0$ somewhere in $M_{i}$ for all $M_{i}$ such that $\partial M_{i} = \overline{\gamma_{i}}$, then $\Theta_{j} <1$ for all $j$ and hence $\Theta <1$.
\end{proof}

In the following, let $\tilde{c}$ be the function in Theorem \ref{thm_comparison}. Define an operator
\[
\widetilde{\mathfrak{L}} v = - \Delta v + \langle \vec{b}, \nabla v \rangle + \tilde{c} v.
\]
on $M$. Suppose $\tilde{c} \ge 0$.
\begin{lemma}\label{lem_operator_comparison}
Suppose $M'$ is a subdomain of $M$. Suppose $\tilde{c} \ge c$. Let $\psi$ be a nonnegative bounded function on $\partial M'$. Let $u_{\psi}$ (resp.~$\tilde{u}_{\psi}$) be the Perron solution to $\mathfrak{L} =0$ (resp.~$\widetilde{\mathfrak{L}} = 0$) with boundary value $\psi$. Then $\tilde{u}_{\psi}|_{M' \setminus \partial M'} < u_{\psi}|_{M' \setminus \partial M'}$ or $\tilde{u}_{\psi}|_{M' \setminus \partial M'} = u_{\psi}|_{M' \setminus \partial M'}$.
\end{lemma}
\begin{proof}
Let $S$ (resp.~$\tilde{S}$) be the set consisting of subfunctions of $\mathfrak{L} =0$ (resp.~$\widetilde{\mathfrak{L}} = 0$) with boundary value $\psi$ (see $\S$\ref{subsec_Perron}). Since $\psi \ge 0$ and $\tilde{c} \ge c \ge 0$, we have $0 \in S$ and $0 \in \tilde{S}$. Thus $u_{\psi} = \sup_{v \in S^{+}} v$ and $\tilde{u}_{\psi} = \sup_{v \in \tilde{S}^{+}} v$, where $S^{+} = \{ v \ge 0 \mid v \in S \}$ and $\tilde{S}^{+} = \{ v \ge 0 \mid v \in \tilde{S} \}$.

We first prove $u_{\psi} \ge \tilde{u}_{\psi}$. It suffices to show that $\tilde{S}^{+} \subseteq S^{+}$. Suppose $v \in \tilde{S}^{+}$. Let $x$ be a point in $M' \setminus \partial M'$. Suppose $B$ is a neighborhood of $x$ in $M' \setminus \partial M'$ such that it is identified with a closed ball with center $x$ in a local coordinate chart. Let $\tilde{w}$ be the function on $B$ which solves the problem $\widetilde{\mathfrak{L}} \tilde{w} =0$ and $\tilde{w}|_{\partial B} =v$. Then $0 \le v \le \tilde{w}$ on $B$ (see \cite[p.~102]{Gilbarg_Trudinger}). Since $\tilde{c} \ge c$, we infer $\mathfrak{L} \tilde{w} \le 0$. Let $w$ be the function on $B$ which solves the problem $\mathfrak{L} w =0$ and $w|_{\partial B} =v$. By Corollary \ref{cor_comparison}, we see $w \ge \tilde{w}$, and then $w \ge v$ on $B$. We further infer $v \in S^{+}$ (see \cite[p.~102]{Gilbarg_Trudinger}). Thus $\tilde{S}^{+} \subseteq S^{+}$ and $0 \le \tilde{u}_{\psi} \le u_{\psi}$.

In $M' \setminus \partial M'$, we have $\mathfrak{L} u_{\psi}=0$, $\widetilde{\mathfrak{L}} \tilde{u}_{\psi} =0$ and $\tilde{u}_{\psi} - u_{\psi} \le 0$. Since $\tilde{u}_{\psi} \ge 0$ and $\tilde{c} \ge c$, we infer $\mathfrak{L} \tilde{u}_{\psi} \le 0$ and then $\mathfrak{L} (\tilde{u}_{\psi} - u_{\psi}) \le 0$. The proof is finished by the strong maximum principle \ref{thm_maximum}, .
\end{proof}

\begin{proof}[Proof of Theorem \ref{thm_comparison}]
For brevity, when we compare $(c, \mathcal{D}, \rho)$ with $(\tilde{c}, \mathcal{D}, \rho)$, let $\Theta (c)$ (resp.~$\Theta (\tilde{c})$) denote $\Theta (c, \mathcal{D}, \rho)$ (resp.~$\Theta (\tilde{c}, \mathcal{D}, \rho)$), let $\Theta_{i} (c)$ (resp.~$\Theta_{i} (\tilde{c})$) denote $\Theta_{i} (c, \mathcal{D}, \rho)$ (resp.~$\Theta_{i} (\tilde{c}, \mathcal{D}, \rho)$), let $\theta_{i} (c)$ (resp.~$\theta_{i} (\tilde{c})$) denote $\theta_{i} (c, \mathcal{D}, \rho)$ (resp.~$\theta_{i} (\tilde{c}, \mathcal{D}, \rho)$).

Similar abbreviations are used for comparing $(c, \mathcal{D}, \rho)$ with $(c, \widetilde{\mathcal{D}}, \rho)$.

(1). Recall that $\theta_{i}$ is a function, and $\Theta_{i} = \max \theta_{i}|_{\mathrm{supp} \rho_{i}}$. Since $\theta_{i} =0$ on $\mathrm{supp} \rho_{i} \cap \partial M_{i}$, we also have $\Theta_{i} = \max \theta_{i}|_{\mathrm{supp} \rho_{i} \setminus \partial M_{i}}$.

Since $\tilde{c} \ge c$, by Lemma \ref{lem_operator_comparison}, we have $\theta_{i} (\tilde{c}) \le \theta_{i} (c)$ in $M_{i} \setminus \partial M_{i}$ for all $i$. Thus $\Theta_{i} (\tilde{c}) \le \Theta_{i} (c)$ for all $i$ and hence $\Theta (\tilde{c}) \le \Theta (c)$.

(2). It suffices to show $\theta_{i} (\tilde{c}) < \theta_{i} (c)$ in $M_{i} \setminus \partial M_{i}$ for all $i$. If not for some $i$, then by Lemma \ref{lem_operator_comparison}, we would have $\theta_{i} (\tilde{c}) = \theta_{i} (c)$ in $M_{i} \setminus \partial M_{i}$. On the other hand, by the assumption and the continuity of $\tilde{c}$ and $c$, we have $\tilde{c} (x) > c (x)$ for some $x \in M_{i} \setminus \partial M_{i}$. By (3) of Lemma \ref{lem_bound_perron} (replacing $c$ with $\tilde{c}$), $\theta_{i} (\tilde{c})|_{M_{i} \setminus \partial M_{i}} >0$. Then we would have
\[
\widetilde{\mathfrak{L}} \theta_{i} (\tilde{c}) (x) > \mathfrak{L} \theta_{i} (\tilde{c}) (x) = \mathfrak{L} \theta_{i} (c) (x) =0,
\]
which is a contradiction. (Here $\theta_{i} (\tilde{c}) (x)$ and $\theta_{i} (c) (x)$ are the values of $\theta_{i} (\tilde{c})$ and $\theta_{i} (c)$ at $x$, respectively. )

(3). We first prove, $\forall x \in M_{i} \setminus \partial M_{i}$, we have $\lim\limits_{c \rightarrow +\infty} \theta_{i} (c) (x) =0$.

Choosing a local coordinate chart, we may assume $x$ is identified with $0 \in \mathbb{R}^{d}$, and a neighborhood of $x$ in this chart is identified with the closed unit ball $B$ in $\mathbb{R}^{d}$. Let $\alpha$ be a positive constant. Define $w(\xi_{1}, \dots, \xi_{d}) = e^{\alpha r^{2} - \alpha}$, where $r = (\sum_{i=1}^{d} \xi_{i}^{2})^{\frac{1}{2}}$ and $(\xi_{1}, \dots, \xi_{d})$ are the coordinates of $\mathbb{R}^{d}$. Then $w>0$ and $w|_{\partial B} =1$. Since
\[
- \Delta w + \langle \vec{b}, \nabla w \rangle
\]
is continuous and hence bounded on $B$, and since $w$ is continuous and positive, we infer there exists a constant $c_{0} >0$, $\forall c> c_{0}$, we have
\[
\mathfrak{L} w = - \Delta w + \langle \vec{b}, \nabla w \rangle + cw >0
\]
on $B$. Note that $\mathfrak{L} \theta_{i} (c) =0$. By (5) of Lemma \ref{lem_bound_perron}, we see $\theta_{i} (c)|_{\partial B} <1$ since $c>0$. So $\mathfrak{L} (\theta_{i} (b) -w) < 0$ on $B$, and $\theta_{i} (b) -w <0$ on $\partial B$. By Corollary \ref{cor_comparison}, we have $\theta_{i} (c) <w$ on $B$. Particularly, $\theta_{i} (c) (x) < w(0) = e^{-\alpha}$.

In summary, $\forall \alpha >0$, $\exists c_{0} \in (0,+\infty)$, $\forall c> c_{0}$, we have $\theta_{i} (c) (x) < e^{-\alpha}$. So $\lim\limits_{c \rightarrow +\infty} \theta_{i} (c) (x) =0$.

We also know that $\theta_{i} (c)|_{\partial M_{i} \setminus \overline{\gamma_{i}}} =0$. So $\theta_{i} (c)$ converges to $0$ in $M_{i} \setminus \overline{\gamma_{i}}$ when $c \rightarrow +\infty$. By the proof of (1), this convergence is decreasing. By (1) of Lemma \ref{lem_bound_perron}, $\theta_{i} (c)$ is continuous on $M_{i} \setminus \overline{\gamma_{i}}$.  Since $\mathrm{supp} \rho_{i} \subset M_{i} \setminus \overline{\gamma_{i}}$ and $\mathrm{supp} \rho_{i}$ is compact, by Dini's Theorem, the convergence of $\theta_{i} (c)$ is uniform on $\mathrm{supp} \rho_{i}$. Thus $\Theta_{i} (c)$ and hence $\Theta (c)$ converge to $0$ when $c \rightarrow +\infty$.

(4). We only need to show $\theta_{i} (\widetilde{\mathcal{D}})|_{M_{i}} \le \theta_{i} (\mathcal{D})$ for all $i$.

Let $S(\mathcal{D})$ and $S(\widetilde{\mathcal{D}})$ be the sets of subfunctions defining $\theta_{i} (\mathcal{D})$ and $\theta_{i} (\widetilde{\mathcal{D}})$, respectively (see $\S$\ref{subsec_Perron}). It suffices to show $\forall v \in S(\widetilde{\mathcal{D}})$, we have $v|_{M_{i}} \in S(\mathcal{D})$.

Clearly, $v|_{M_{i}} \in S(\mathcal{D})$ is a subsolution to $\mathfrak{L} =0$. We also know $0 \le \theta_{i} (\widetilde{\mathcal{D}}) \le 1$. Since $\widetilde{M}_{i} \supseteq M_{i}$, we have $\theta_{i} (\widetilde{\mathcal{D}})|_{\overline{\gamma_{i}}} \le 1$. By Lemma \ref{lem_subdomain_interior}, we have $\partial M_{i} \setminus \overline{\gamma_{i}} \subseteq \partial \widetilde{M}_{i} \setminus \overline{\widetilde{\gamma}_{i}}$. We also have $\theta_{i} (\widetilde{\mathcal{D}})|_{\partial M_{i} \setminus \overline{\gamma_{i}}} =0$. Since $v \le \theta_{i} (\widetilde{\mathcal{D}})$, we further infer $v|_{M_{i}} \in S(\mathcal{D})$.

(5). By (1) of Lemma \ref{lem_bound_perron} (replacing $\mathcal{D}$ with $\widetilde{\mathcal{D}}$) and the fact $\widetilde{M}_{i} \setminus \overline{\widetilde{\gamma}_{i}} \supseteq M_{i}$, we know $\theta_{i} (\widetilde{\mathcal{D}})|_{M_{i}}$ is a continuous Perron solution on $M_{i}$.

By the proof of (4), we know $\theta_{i} (\widetilde{\mathcal{D}})|_{M_{i}} \le \theta_{i} (\mathcal{D})$. Since $\Theta (\mathcal{D}) <1$, we infer $\theta_{i} (\widetilde{\mathcal{D}}) \le \theta_{i} (\mathcal{D}) <1$ on $\mathrm{supp} \rho_{i}$. By (3) of Lemma \ref{lem_bound_perron} (replacing $\mathcal{D}$ with $\widetilde{\mathcal{D}}$), we further infer $\theta_{i} (\widetilde{\mathcal{D}}) <1$ in $\widetilde{M}_{i} \setminus \overline{\widetilde{\gamma}_{i}}$. Thus $\theta_{i} (\widetilde{\mathcal{D}}) \le \theta_{i} (\mathcal{D})$ on $\partial M_{i}$ and $\theta_{i} (\widetilde{\mathcal{D}}) \ne \theta_{i} (\mathcal{D})$ on $\gamma_{i}$.

By (1) of Lemma \ref{lem_bound_perron} again, $\theta_{i} (\mathcal{D})$ is also continuous at each point in $\gamma_{i}$. By the strong maximum principle \ref{cor_maximum}, $\theta_{i} (\widetilde{\mathcal{D}}) < \theta_{i} (\mathcal{D})$ in $M_{i} \setminus \partial M_{i}$ and hence $\Theta_{i} (\widetilde{\mathcal{D}}) < \Theta_{i} (\mathcal{D})$, which implies $\Theta (\widetilde{\mathcal{D}}) < \Theta (\mathcal{D})$.
\end{proof}

\begin{proof}[Proof of Theorem \ref{thm_refine}]
(1). By the definition of $\Theta (\mathcal{D}, \rho)$, there exist an $M_{i} \in \mathcal{D}$ and an $x \in \mathrm{supp} \rho_{i}$ such that $\Theta (\mathcal{D}, \rho) = \theta_{i} (\mathcal{D}, \rho) (x)$. Since $(\widetilde{\mathcal{D}}, \tilde{\rho})$ is a refinement, there exists an $\widetilde{M}_{j} \in \widetilde{\mathcal{D}}$ such that $\widetilde{M}_{j} \subseteq M_{i}$ and $x \in \mathrm{supp} \tilde{\rho}_{j}$.

Since $\widetilde{M}_{j} \subseteq M_{i}$, as can be seen in the proof of the (4) in Theorem \ref{thm_comparison}, we have $\theta_{i} (\mathcal{D}, \rho)|_{\widetilde{M}_{j}} \le \theta_{j} (\widetilde{\mathcal{D}}, \tilde{\rho})$. (Notice that we need to take $\widetilde{M}_{i}$ and $M_{i}$ there as $M_{i}$ and $\widetilde{M}_{j}$ here, respectively.) Thus
\[
\Theta (\mathcal{D}, \rho) = \theta_{i} (\mathcal{D}, \rho) (x) \le \theta_{j} (\widetilde{\mathcal{D}}, \tilde{\rho}) (x) \le \Theta_{j} (\widetilde{\mathcal{D}}, \tilde{\rho}) \le \Theta (\widetilde{\mathcal{D}}, \tilde{\rho}).
\]

(2). Again, $\Theta (\mathcal{D}, \rho) = \theta_{i} (\mathcal{D}, \rho) (x)$ for some $x \in \mathrm{supp} \rho_{i} \setminus \partial M$. Since $(\widetilde{\mathcal{D}}, \tilde{\rho})$ is a strict refinement, there exists an $\widetilde{M}_{j} \in \widetilde{\mathcal{D}}$ such that $\widetilde{M}_{j} \subsetneqq M_{i}$ and $x \in \mathrm{supp} \tilde{\rho}_{j} \setminus \partial M$.

We claim $\partial \widetilde{M}_{j} \cap (M_{i} \setminus \partial M_{i}) \ne \emptyset$. Otherwise, $\widetilde{M}_{j} \setminus \partial \widetilde{M}_{j}$ would be relatively both open and closed in $M_{i} \setminus \partial M_{i}$. Since $M_{i} \setminus \partial M_{i}$ is connected, we would have $\widetilde{M}_{j} \setminus \partial \widetilde{M}_{j} = M_{i} \setminus \partial M_{i}$ and hence $\widetilde{M}_{j} = M_{i}$, which is a contradiction.

Now we can choose a
\[
y \in \widetilde{\gamma}_{j} \cap (M_{i} \setminus \partial M_{i}) = \partial \widetilde{M}_{j} \cap (M_{i} \setminus \partial M_{i}).
\]
By the assumption, $\Theta (\mathcal{D}, \rho) < 1$. So $\theta_{i} (\mathcal{D}, \rho) <1$ in $\mathrm{supp} \rho_{i}$. By (3) of Lemma \ref{lem_bound_perron}, we know $\theta_{i} (\mathcal{D}, \rho) <1$ in $M_{i} \setminus \partial M_{i}$ and hence $\theta_{i} (\mathcal{D}, \rho) (y) <1$.

We already know $\theta_{i} (\mathcal{D}, \rho)|_{\widetilde{M}_{j}} \le \theta_{j} (\widetilde{\mathcal{D}}, \tilde{\rho})$. By (1) of Lemma \ref{lem_bound_perron}, both $\theta_{i} (\mathcal{D}, \rho)$ and $\theta_{j} (\widetilde{\mathcal{D}}, \tilde{\rho})$ are continuous at $y$, and $\theta_{i} (\mathcal{D}, \rho) (y) <1 = \theta_{j} (\widetilde{\mathcal{D}}, \tilde{\rho}) (y)$. By the strong maximum principle \ref{cor_maximum} again, we have $\theta_{i} (\mathcal{D}, \rho) < \theta_{j} (\widetilde{\mathcal{D}}, \tilde{\rho})$ in $\widetilde{M}_{j} \setminus \partial \widetilde{M}_{j}$. Since
\[
x \in \mathrm{supp} \tilde{\rho}_{j} \setminus \partial M \subseteq \widetilde{M}_{j} \setminus \partial \widetilde{M}_{j},
\]
we have $\theta_{i} (\mathcal{D}, \rho) (x) < \theta_{j} (\widetilde{\mathcal{D}}, \tilde{\rho}) (x)$ and $\Theta (\mathcal{D}, \rho) < \Theta (\widetilde{\mathcal{D}}, \tilde{\rho})$.
\end{proof}

\section{$C^{0}$-Convergence}\label{sec_convergence}
We prove Theorems \ref{thm_rate}, \ref{thm_rate_boundary}, and \ref{thm_slow} in this section.

\subsection{First Step}
Let's consider Theorem \ref{thm_rate} first.

Define $C^{0}_{0} (M) := \{ v \in C^{0} (M) \mid v|_{\partial M} =0 \}$ as usual, where
$C^{0}_{0} (M) = C^{0} (M)$ if $\partial M = \emptyset$. Define a map $T \colon C^{0}_{0} (M) \rightarrow C^{0}_{0} (M)$ as follows. Suppose $v \in C^{0}_{0} (M)$. For each subdomain $M_{i}$, solve the problem
\begin{equation}\label{eqn_error_operator}
\left\{
\begin{aligned}
\mathfrak{L} w_{i}  & = 0, & \text{in $M_{i} \setminus \partial M_{i}$;}
\\
w_{i} & = v (x), & \text{on $\partial M_{i}$.}
\end{aligned}
\right.
\end{equation}
Define $Tv := w = \sum_{i=1}^{m} \rho_{i} w_{i}$.

\begin{lemma}\label{lem_error_operator}
The map $v \mapsto Tv$ is a well-defined linear map $T \colon C^{0}_{0} (M) \rightarrow C^{0}_{0} (M)$.
\end{lemma}
\begin{proof}
Since $v|_{\partial M_{i}}$ is continuous and $\partial M_{i}$ is Lipschitz, by Lemma \ref{lem_perron_lipschitz}, the $w_{i}$ is a well-defined element in $C^{0} (M_{i})$. Since $v|_{\partial M} =0$, we see $w_{i}|_{\partial M} =0$ too. So $(Tv)|_{\partial M} =0$. The linearity of $\mathfrak{L}$ implies that of the map $v \mapsto w_{i}$. Thus $T$ is linear.

It remains to prove $Tv$ is continuous in $M$. To do so, we only need to prove the continuity of $\rho_{i} w_{i}$. By Lemma \ref{lem_subdomain_interior}, we know $M_{i} \setminus \overline{\gamma_{i}}$ and $M \setminus \mathrm{supp} \rho_{i}$ form an open cover of $M$. In $M_{i} \setminus \overline{\gamma_{i}}$, both $w_{i}$ and $\rho_{i}$ are continuous, thus so is $\rho_{i} w_{i}$. In $M \setminus \mathrm{supp} \rho_{i}$, by definition, $\rho_{i} w_{i} =0$, which is obviously continuous. In summary, $\rho_{i} w_{i}$ in continuous in $M$.
\end{proof}

The importance of $T$ lies in the following proposition.
\begin{proposition}\label{prop_error_operator}
Let $u^{0} \in C^{0} (M)$ be an arbitrary initial guess in Algorithm \ref{alg_continuous_parallel}, let $u^{n}$ be the one generated by $u^{0}$ there. Then, $\forall n > 0$, we have $(u^{n} - u) \in C^{0}_{0} (M)$, $u^{n} - u = T (u^{n-1} - u)$ and $u^{n} - u = T^{n} (u^{0} - u)$.
\end{proposition}
\begin{proof}
By Lemma \ref{lem_algorithm_wellpose}, we know $(u^{n} - u) \in C^{0}_{0} (M)$. Furthermore,
\begin{equation}\label{prop_error_operator_1}
u^{n} - u = \sum_{i=1}^{m} \rho_{i} u^{n}_{i} -u = \sum_{i=1}^{m} \rho_{i} (u^{n}_{i} -u),
\end{equation}
$\mathfrak{L} (u^{n}_{i} -u) =0$ in $M_{i} \setminus \partial M_{i}$ and $(u^{n}_{i} -u)|_{\partial M_{i}} = u^{n-1} -u$. The conclusion follows.
\end{proof}

Let $\theta_{i}$ be the function in \eqref{eqn_bound_function}. Let $\Theta$ be the one in Theorem \ref{thm_bound}.
\begin{lemma}\label{lem_bound_function}
The function $\tau := \sum_{i=1}^{m} \rho_{i} \theta_{i}$ is a well defined element in $C^{0}_{0} (M)$, and $0 \le \tau \le \Theta$. Furthermore, $\tau >0$ in $M \setminus \partial M$.
\end{lemma}
\begin{proof}
By (1) of Lemma \ref{lem_bound_perron}, we know $\theta_{i}$ is continuous in $M_{i} \setminus \overline{\gamma_{i}}$. By the same argument on the continuity of $Tv$ in Lemma \ref{lem_error_operator}, we know $\tau$ is continuous. Since $\theta_{i} =0$ on $\partial M_{i} \setminus \overline{\gamma_{i}}$, we further see $\tau \in C^{0}_{0} (M)$.

By (3) of Lemma \ref{lem_bound_perron}, we know $\theta_{i} >0$ in $M_{i} \setminus \partial M_{i}$. Since $M_{i} \cap \partial M = \partial M_{i} \cap \partial M$, $\gamma_{i} = \partial M_{i} \setminus \partial M$ and $\overline{\gamma_{i}} \subseteq \partial M_{i}$, we have
\begin{equation}\label{lem_bound_function_1}
\mathrm{supp} \rho_{i} \setminus \partial M \subset (M_{i} \setminus \overline{\gamma_{i}}) \setminus \partial M = M_{i} \setminus [\overline{\gamma_{i}} \cup (M_{i} \cap \partial M)] = M_{i} \setminus \partial M_{i}.
\end{equation}
So $\tau >0$ in $M \setminus \partial M$. Since $\Theta_{i} = \max_{\mathrm{supp} \rho_{i}} \theta_{i}$ and $\Theta = \max\{\Theta_{1}, \dots, \Theta_{m} \}$, we infer
\[
\tau = \sum_{i=1}^{m} \rho_{i} \theta_{i} \le \sum_{i=1}^{m} \rho_{i} \Theta_{i} \le \sum_{i=1}^{m} \rho_{i} \Theta = \Theta,
\]
which finishes the proof.
\end{proof}

\begin{proposition}\label{prop_1_iteration_bound}
If $v \in C^{0}_{0} (M)$, then
\begin{equation}\label{prop_1_iteration_bound_1}
-\| v \|_{C^{0} (M)} \tau \le Tv \le \| v \|_{C^{0} (M)} \tau.
\end{equation}
Particularly, $T$ is a continuous linear map with $\| T \| \le \Theta$.
\end{proposition}
\begin{proof}
By (3) of Lemma \ref{lem_perron}, $\| v \| \theta_{i}$ is the Perron solution to $\mathfrak{L} =0$ in $M_{i}$ with prescribed boundary value: $\| v \|$ on $\overline{\gamma_{i}}$ and $0$ on $\partial M_{i} \setminus \overline{\gamma_{i}}$. Let $w_{i}$ be the one obtained from $v$ via \eqref{eqn_error_operator}. By (4) of Lemma \ref{lem_perron}, we have $w_{i} \le \| v \| \theta_{i}$. So
\[
Tv = \sum_{i=1}^{m} \rho_{i} w_{i} \le \sum_{i=1}^{m} \rho_{i} \| v \| \theta_{i} = \| v \| \tau.
\]
We further infer $-Tv = T(-v) \le \| -v \| \tau = \| v \| \tau$. In summary, we obtain \eqref{prop_1_iteration_bound_1}.

Finally, $\| T \| \le \Theta$ follows from \eqref{prop_1_iteration_bound_1} and Lemma \ref{lem_bound_function}.
\end{proof}

We are ready to prove Theorem \ref{thm_rate}.
\begin{proof}[Proof of Theorem \ref{thm_rate}]
(1). By \eqref{prop_error_operator_1}, we have, $\forall x \in M$,
\begin{align*}
|(u^{n} - u) (x)| & \le \sum_{i=1}^{m} \rho_{i} (x) |(u^{n}_{i} -u) (x)| \le \sum_{i=1}^{m} \rho_{i} (x) \| u^{n}_{i} -u \|_{C^{0} (M_{i})} \\
& \le \max_{1 \le i \le m} \{ \| u^{n}_{i} - u \|_{C^{0} (M_{i})} \} \sum_{i=1}^{m} \rho_{i} (x) = \max_{1 \le i \le m} \{ \| u^{n}_{i} - u \|_{C^{0} (M_{i})} \}.
\end{align*}
Note that, if $x \notin M_{i}$, then $(u^{n}_{i} -u) (x)$ is undefined; however, $\rho_{i} (x) =0$ and the summand $\rho_{i} (x) |(u^{n}_{i} -u) (x)|$ is redundant in this case. So the above estimate makes sense and $\| u^{n} - u \| \le \max_{1 \le i \le m} \{ \| u^{n}_{i} - u \|_{C^{0} (M_{i})} \}$.

Furthermore, $\mathfrak{L} (u^{n}_{i} -u) =0$ and $(u^{n}_{i} -u) \in C^{0} (M_{i})$. By the weak maximum principle \ref{cor_maximum}, we have
\begin{align*}
\| u^{n}_{i} - u \|_{C^{0}(M_{i})} & = \| (u^{n}_{i} - u)|_{\partial M_{i}} \|_{C^{0}(\partial M_{i})} = \| (u^{n-1} - u)|_{\partial M_{i}} \|_{C^{0}(\partial M_{i})} \\
& \le \| u^{n-1} - u \|_{C^{0}(M)},
\end{align*}
which implies $\max_{1 \le i \le m} \{ \| u^{n}_{i} - u \|_{C^{0} (M_{i})} \} \le \| u^{n-1} - u \|$.

(2). This follows from Propositions \ref{prop_error_operator} and \ref{prop_1_iteration_bound}.
\end{proof}

\subsection{With Boundary}
We now prove Theorems \ref{thm_rate_boundary} and \ref{thm_slow}. We assume $\partial M \ne \emptyset$ throughout this subsection.

The operator $T$ and function $\tau$ below are the ones defined as above.
\begin{lemma}\label{lem_error_operator_comparison}
Suppose $v_{j} \in C^{0}_{0} (M)$ ($j=1,2$) and $v_{1} \le v_{2}$. Then $T v_{1} \le T v_{2}$.
\end{lemma}
\begin{proof}
Let $w_{ij} \in C^{0} (M_{i})$ be the function obtained from $v_{j}$ via \eqref{eqn_error_operator}. Then
\[
w_{i1}|_{\partial M_{i}} = v_{1}|_{\partial M_{i}} \le v_{2}|_{\partial M_{i}} = w_{i2}|_{\partial M_{i}}.
\]
Since $w_{ij}$'s are solutions to $\mathfrak{L} =0$, by Corollary \ref{cor_comparison}, we have $w_{i1} \le w_{i2}$. Now the conclusion follows from the definition of $T$.
\end{proof}

Proposition \ref{prop_1_iteration_bound} and Lemma \ref{lem_error_operator_comparison} immediately imply the following result.
\begin{proposition}\label{prop_n_iteration_bound}
If $v \in C^{0}_{0} (M)$, then, $\forall n >0$, $T^{n} v \in C^{0}_{0} (M)$ and
\[
-\| v \|_{C^{0} (M)} T^{n-1} \tau \le T^{n}v \le \| v \|_{C^{0} (M)} T^{n-1} \tau.
\]
Particularly, $\| T^{n} \| \le \| T^{n-1} \tau \|_{C^{0} (M)}$ for all $n>0$.
\end{proposition}

As a result, the estimate of $\| T^{n} \|$ is reduced to that of $T^{n-1} \tau$. Recall that we let $U_{i}$ denote $\{ x \in M \mid \rho_{i} (x) >0 \}$.
\begin{proposition}\label{prop_iteration_function}
Let $N$ be the winding number. Then $0 < T^{n} \tau \le 1$ in $M \setminus \partial M$, for all $n \ge 0$, and $T^{n} \tau < 1$ for all $n \ge N-1$.
\end{proposition}
\begin{proof}
We first note that $0 < T^{n} \tau \le 1$ on $M \setminus \partial M$. By Lemma \ref{lem_bound_function}, this already holds for $n=0$. Using the strong maximum principle \ref{cor_maximum} and \eqref{lem_bound_function_1}, the general statement is easily verified by induction on $n$.

In the following, we prove $T^{n} \tau < 1$ on $\bigcup_{i \in I_{n+1}} U_{i}$ by induction on $n$.

Take $n=0$. By the definition of $I_{n+1} = I_{1}$, we have, $\forall i \in I_{1}$, $\partial M_{i} \ne \overline{\gamma_{i}}$, and $\theta_{i} |_{U_{i}} < 1$ by (4) of Lemma \ref{lem_bound_perron}. We also know $\rho_{j}|_{U_{j}} >0$ and $\theta_{j} \le 1$ for each $j$. So
\[
T^{0} \tau = \tau = \sum_{j=1}^{m} \rho_{j} \theta_{j} < \sum_{j=1}^{m} \rho_{j} = 1
\]
on $\bigcup_{i \in I_{1}} U_{i}$. The claim is true for $n=0$.

Assume the claim is proved for $n=k$. Let's consider the case of $n=k+1$. Since $0 \le T^{k} \tau \le 1$, by (2) of Lemma \ref{lem_perron}, $0 \le w_{j} \le 1$ for all $j$, where $w_{j}$ is obtained from $T^{k} \tau$ via \eqref{eqn_error_operator}.

Suppose $i \in I_{k+2}$. By Lemma \ref{lem_winding}, we have $\partial M_{i} \ne \overline{\gamma_{i}}$ or $\partial M_{i} \cap (\bigcup_{j \in I_{k+1}} U_{j}) \ne \emptyset$. If $\partial M_{i} \ne \overline{\gamma_{i}}$, then $T^{k} \tau =0$ on $\partial M_{i} \setminus \overline{\gamma_{i}} \subseteq \partial M$. If $\partial M_{i} \cap (\bigcup_{j \in I_{k+1}} U_{j}) \ne \emptyset$, by inductive hypothesis, $T^{k} \tau <1$ on $\partial M_{i} \cap (\bigcup_{j \in I_{k+1}} U_{j})$. By the strong maximum principle \ref{cor_maximum}, in either case, we have $w_{i} <1$ on $M_{i} \setminus \overline{\gamma_{i}}$.

Since $U_{j} \subset M_{j} \setminus \overline{\gamma_{j}}$ and $\rho_{j}|_{U_{j}} >0$ for each $j$, we have, on $\bigcup_{i \in I_{k+2}} U_{i}$,
\[
T^{k+1} \tau = \sum_{j=1}^{m} \rho_{j} w_{j} < \sum_{j=1}^{m} \rho_{j} = 1.
\]
The claim is verified for $n=k+1$ and the induction is completed.

By Proposition \ref{prop_winding}, we have $I_{n+1} = \{ i \mid 1 \le i \le m\}$ for all $n \ge N-1$. Now the conclusion follows from the fact $\bigcup_{i=1}^{m} U_{i} = M$.
\end{proof}

\begin{remark}
The proof of Proposition \ref{prop_iteration_function} is an inductive argument on a sequence of subdomains, starting from those adjacent to $\partial M$ and successively moving toward the interior. The same type of argument was used in \cite{lions3} and subsequently in \cite{deng2}, \cite{deng3}, and \cite{qin_xu06}.
\end{remark}

\begin{remark}
Proposition \ref{prop_iteration_function} also relies on Proposition \ref{prop_winding}, which involves intricate topological arguments.
\end{remark}

We are at a position to prove Theorem \ref{thm_rate_boundary}.
\begin{proof}[Proof of Theorem \ref{thm_rate_boundary}]
(1). Let $L := \| T^{N-1} \tau \|_{C^{0} (M)}^{\frac{1}{N}}$. By Propositions \ref{prop_iteration_function} and \ref{prop_n_iteration_bound}, we have $L \in (0,1)$ and $\| T^{N} \| \le L^{N}$. Since $(u^{n+N} - u) = T^{N} (u^{n} -u)$ by Proposition \ref{prop_error_operator}, the conclusion follows.

(2). We already know that $\| T^{N} \| \le L^{N}$ and $L>0$. For each $n \ge 0$, $n= N [\frac{n}{N}] +i$ for some $0 \le i < N$. Thus
\[
\| T^{n} \| \le \| T^{N} \|^{[\frac{n}{N}]} \| T \|^{i} \le L^{N[\frac{n}{N}]} \| T \|^{i} = (\| T \|^{i} L^{-i}) L^{n}.
\]
Let $C := \max \{ \| T \|^{i} L^{-i} \mid 0 \le i <N \}$, then $\| T^{n} \| \le C L^{n}$ for all $n \ge 0$. The conclusion now follows from Proposition \ref{prop_error_operator}.
\end{proof}

\begin{remark}
By the above proof, we see the $L$ in Theorem \ref{thm_rate_boundary} has an explicit expression $\| T^{N-1} \tau \|_{C^{0} (M)}^{\frac{1}{N}}$.
\end{remark}

The ingredient of Theorem \ref{thm_slow} is the following lemma whose proof is an argument ``dual" to that of Proposition \ref{prop_iteration_function}.
\begin{lemma}\label{lem_slow}
Suppose $N>1$, and $c=0$ in \eqref{eqn_operator}. Suppose $M \ne \bigcup_{i \ne j} M_{i}$ for each $j$. Suppose $v \in C^{0}_{0} (M)$ such that $0 \le v \le 1$ and $v=1$ on $\bigcup_{i \notin I_{1}} M_{i}$ (resp.~$\bigcup_{i \notin I_{2}} M_{i}$). Then $\| T^{n} v \| =1$ for $n \le N-1$ (resp.~$n \le N-2$).
\end{lemma}
\begin{proof}
We first consider the case that $v=1$ on $\bigcup_{i \notin I_{1}} M_{i}$. Let's verify that $0 \le T^{n} v \le 1$ and $T^{n} v =1$ on $M \setminus \bigcup_{i \in I_{n}} U_{i}$ by induction on $n$.

Let $w_{i}$ be obtained from $v$ via \eqref{eqn_error_operator}. Since $0 \le v \le 1$, we have $0 \le w_{i} \le 1$ for all $i$ and hence $0 \le Tv \le 1$. Furthermore, by the assumption, $v|_{\partial M_{j}} =1$ for $j \notin I_{1}$ and $c=0$, we infer $w_{j} = 1$ for $j \notin I_{1}$. If $x \notin \bigcup_{i \in I_{1}} U_{i}$, then $\rho_{i} (x) =0$ for each $i \in I_{1}$ and
\[
Tv(x) = \sum_{i=1}^{m} \rho_{i} (x) w_{i} (x) = \sum_{j \notin I_{1}} \rho_{j} (x) w_{j} (x) = \sum_{j \notin I_{1}} \rho_{j} (x) = \sum_{i=1}^{m} \rho_{i} (x) =1.
\]
So the claim is proved when $n=1$.

Suppose the claim is proved for $n=k$. Let's consider the case of $n=k+1$. By inductive hypothesis, we have $0 \le T^{k} v \le 1$. Thus $0 \le w_{ik} \le 1$ for all $i$, where $w_{ik}$ is obtained from $T^{k} v$ via \eqref{eqn_error_operator}. We infer $0 \le T^{k+1} v \le 1$. On the other hand, if $j \notin I_{k+1}$, by Lemma \ref{lem_winding}, we have $\partial M_{j} \cap ( \bigcup_{i \in I_{k}} U_{i} ) = \emptyset$. Also by inductive hypothesis, $T^{k} v =1$ on $M \setminus \bigcup_{i \in I_{k}} U_{i}$. So $T^{k} v = 1$ on $\partial M_{j}$ for $j \notin I_{k+1}$. Since $c=0$, we infer $w_{jk} = 1$ for $j \notin I_{k+1}$. If $x \notin \bigcup_{i \in I_{k+1}} U_{i}$, then $\rho_{i} (x) =0$ for each $i \in I_{k+1}$, and
\[
T^{k+1} v (x) = \sum_{i=1}^{m} \rho_{i} (x) w_{ik} (x) = \sum_{j \notin I_{k+1}} \rho_{j} (x) w_{jk} (x) = \sum_{j \notin I_{k+1}} \rho_{j} (x) = \sum_{i=1}^{m} \rho_{i} (x) =1.
\]
The claim is proved for $n=k+1$ and the induction is completed.

Suppose $n \le N-1$. By the definition of $N$, we know $I_{n+1} \setminus I_{n} \ne \emptyset$. By the assumption, we have $M \ne \bigcup_{i \in I_{n}} M_{i} \supseteq \bigcup_{i \in I_{n}} U_{i}$. Therefore, $T^{n} v =1$ on the nonempty set $M \setminus \bigcup_{i \in I_{n}} U_{i}$. Since $0 \le T^{n} v \le 1$, we infer $\| T^{n} v \| =1$.

For the case that $v=1$ on $\bigcup_{i \notin I_{2}} M_{i}$, the argument is similar. We can show $T^{n} v =1$ on $M \setminus \bigcup_{i \in I_{n+1}} U_{i}$ by induction on $n$.
\end{proof}

\begin{proof}[Proof of Theorem \ref{thm_slow}]
We have $\bigcup_{i \notin I_{2}} M_{i} \ne \emptyset$ since $N>2$. By (2) of Lemma \ref{lem_boundary_neighborhood}, we also know $\partial M \cap \bigcup_{i \notin I_{2}} M_{i} = \emptyset$. Both $\partial M$ and $\bigcup_{i \notin I_{2}} M_{i}$ are closed. Therefore, by Urysohn's Lemma (\cite[p.~115]{Kelley}), there exists a $v \in C^{0}_{0} (M)$ such that $0 \le v \le 1$ and $v =1$ on $\bigcup_{i \notin I_{2}} M_{i}$. Define $u^{0} := u+v$, then $u^{0} \ne u$. By Proposition \ref{prop_error_operator}, we have $u^{n} - u = T^{n} v$ . By Lemma \ref{lem_slow}, we infer \eqref{thm_slow_1} holds for $n<N-1$.

If $M_{i} \cap \partial M = \emptyset$ for every $i$ satisfying $\partial M_{i} = \overline{\gamma_{i}}$, then $\partial M \cap \bigcup_{i \notin I_{1}} M_{i} = \emptyset$. Choose a $v \in C^{0}_{0} (M)$ such that $0 \le v \le 1$ and $v =1$ on $\bigcup_{i \notin I_{1}} M_{i}$. Define $u^{0} := u+v$, also by Lemma \ref{lem_slow}, the \eqref{thm_slow_1} holds for $n \le N-1$.
\end{proof}

\begin{remark}\label{rmk_slow}
Let's drop, in Theorem \ref{thm_slow}, the assumption ``$M \ne \bigcup_{i \ne j} M_{i}$ for each $j$". Then we can still prove \eqref{thm_slow_1} holds for $n<N-2$; moreover, if additionally $M_{i} \cap \partial M = \emptyset$ for every $i$ satisfying $\partial M_{i} = \overline{\gamma_{i}}$, then \eqref{thm_slow_1} holds for $n \le N-2$. Indeed, by (2) of Lemma \ref{lem_winding}, for $i \in I_{N} \setminus I_{N-1}$, we have $\partial M_{i} \cap ( \bigcup_{j \in I_{N-2}} U_{j} ) = \emptyset$, which implies $\bigcup_{j \in I_{N-2}} U_{j} \ne M$. Applying this fact in the proof of Lemma \ref{lem_slow}, we see $\| T^{n} v \| =1$ for $n \le N-2$ (resp.~$n \le N-3$) in that lemma if we drop its assumption ``$M \ne \bigcup_{i \ne j} M_{i}$ for each $j$".
\end{remark}

\section{High Regularity}\label{sec_high}
In this section, we prove Theorems \ref{thm_high_global} and \ref{thm_high_local}.

\subsection{Function Spaces}\label{subsec_function}
Let's first recall some widely used function spaces on compact manifolds.

Suppose $(M,g)$ is a Riemannian manifold. The Riemannian metric $g$ defines its Levi-Civita connection $\mathrm{D}$ on $M$. If $f$ is a $C^{k}$ function $(0 \le k < \infty)$ on $M$, then $\mathrm{D}^{k} f$ is the $k$-th covariant differential of $f$. It is a $(0,k)$-type tensor field over $M$. The magnitude $|\mathrm{D}^{k} f|$ is a continuous function on $M$. If $M$ is compact, $|\mathrm{D}^{k} f|$ attains its maximum, we can therefore define the $C^{k}$-norm of $f$ as
\[
\| f \|_{C^{k}} := \max \{ \|\mathrm{D}^{i} f \|_{C^{0}} \mid 0 \le i \le k \}.
\]
The $C^{k}$-norm depends on $g$, whereas the set $C^{k} (M)$ does not depend on $g$. However, since $M$ is compact, all metrics result in the same uniform structure (see e.g.,~\cite[Chapter~6]{Kelley}) on $C^{k} (M)$. In other words, if a metric $g_{1}$ (resp.~$g_{2}$) defines a $C^{k}$-norm $\| \cdot \|_{1}$ (resp.~$\| \cdot \|_{2}$) on $C^{k} (M)$, then $\| \cdot \|_{1}$ and $\| \cdot \|_{2}$ are equivalent (see \cite[p.~64]{conway}). Therefore, in practice, we may reduce estimates of $C^{k}$-norms to the case of functions on Euclidean domains. The technique of reduction is well-known (see e.g.,~\cite[Theorem~2.9]{Aubin}). Let's recall it.

Since $M$ is compact, it can be covered by finitely many coordinate charts $\{ W_{j} \mid 1 \le j \le m \}$. Choose a $C^{\infty}$ partition of unity $\{ \rho_{j} \mid 1 \le j \le m \}$ subordinate to this cover. Then we have, $\forall f \in C^{k} (M)$,
\[
\| f \|_{C^{k}} \le \sum_{j=1}^{m} \| \rho_{j} f \|_{C^{k}} \le C \| f \|_{C^{k}},
\]
where $C$ is a constant independent of $f$, it only depends on the metric and these $\rho_{j}$'s. In other words, the map $f \mapsto \sum_{j=1}^{m} \| \rho_{j} f \|_{C^{k}}$ defines a norm on $C^{k} (M)$ which is equivalent to the original norm $\| \cdot \|_{C^{k}}$. Since $W_{j}$ has local coordinates $(x_{1}, \dots, x_{d})$, we can identify $W_{j}$ with a Euclidean domain and compute the usual $C^{k}$-norm of $\rho_{j} f$ on this Euclidean domain. More precisely, the usual $C^{k}$-norm is
\[
\| \rho_{j} f \|_{C^{k}_{E}} = \max \left\{ \left\| \tfrac{\partial^{i}}{\partial x_{1}^{i_{1}} \cdots \partial x_{d}^{i_{d}}} (\rho_{j} f) \right\|_{C^{0}} \middle| 1 \le i \le k, i_{1} + \cdots + i_{d} =i \right\},
\]
where $\tfrac{\partial^{i}}{\partial x_{1}^{i_{1}} \cdots \partial x_{d}^{i_{d}}}$ is the usual partial derivative. Since $\mathrm{supp} \rho_{j}$ is compact, the norms $\| \cdot \|_{C^{k}}$ and $\| \cdot \|_{C^{k}_{E}}$ are equivalent on $\mathrm{supp} \rho_{j}$. So, to estimate $\| f \|_{C^{k}}$, it suffices to estimate each $\| \rho_{j} f \|_{C^{k}_{E}}$. Obviously,
\[
\| \rho_{j} f \|_{C^{k}_{E}} \le C \max \left\{ \left\| \tfrac{\partial^{i}}{\partial x_{1}^{i_{1}} \cdots \partial x_{d}^{i_{d}}}  f \right\|_{C^{0} (\mathrm{supp} \rho_{j})} \middle| 1 \le i \le k, i_{1} + \cdots + i_{d} =i \right\},
\]
where $C$ is a constant depending only on $\rho_{j}$. As a result, we only need to estimate the partial derivatives of $f$ with respect to these specific coordinates $(x_{1}, \dots, x_{d})$. The theory of PDEs on Euclidean domains can then be applied.

Similarly, for $0 < \alpha \le 1$, we can define the $C^{k, \alpha}$-norm (see e.g.,~\cite[p.~36]{Aubin})
\[
\| f \|_{C^{k, \alpha}} := \| f \|_{C^{k}} + \sup_{p \ne q} \left\{ \frac{|\mathrm{D}^{k} f(p) - \mathrm{D}^{k} f(q)|}{d(p,q)^{\alpha}} \right\},
\]
where $d(p,q)$ is the distance between $p$ and $q$. This results in the space $C^{k, \alpha} (M)$. We also define $C^{k,0} (M)$ as $C^{k} (M)$. For $1 \le p < \infty$, we can also define the Sobolev norm (cf.~\cite[Definition~2.3]{Aubin})
\[
\| f \|_{W^{k,p}} := \left( \sum_{i=0}^{k} \| \mathrm{D}^{i} f\|_{L^{p}}^{p} \right)^{\frac{1}{p}} \qquad \text{and} \qquad \| f \|_{W^{k,\infty}} := \max \{ \| \mathrm{D}^{i} f\|_{L^{\infty}} \mid 0 \le i \le k \},
\]
which results in the space $W^{k,p} (M)$ for $1 \le p \le \infty$. We also denote $W^{k,2} (M)$ by $H^{k} (M)$, which is a Hilbert space. Similar to the case of $\| f \|_{C^{k}}$, the estimates of $\| f \|_{C^{k, \alpha}}$ and $\| f \|_{W^{k,p}}$ can also be reduced to estimates on Euclidean domains.

\subsection{Elliptic Regularity}
The proofs of Theorems \ref{thm_high_global} and \ref{thm_high_local} are essentially applications of the classical elliptic regularity theory.

\begin{proof}[Proof of Theorem \ref{thm_high_global}]
First of all, we show that $u^{n}_{i} - u \in C^{\infty} (M_{i} \setminus \overline{\gamma_{i}})$ for all $n$ and $i$. Since $\mathfrak{L}$ is locally an elliptic operator with $C^{\infty}$ coefficients, this is nothing but a result of interior regularity. Since $\mathfrak{L} (u^{n}_{i} - u) =0$ in the distributional sense, by \cite[Theorem~7.4.1]{Hormander}, we already know $u^{n}_{i} - u \in C^{\infty} (M_{i} \setminus \partial M_{i})$ and $\mathfrak{L} (u^{n}_{i} - u) =0$ in the classical sense. If $\partial M_{i} = \overline{\gamma_{i}}$, there is nothing to prove. We pay more attention to the case that $\partial M_{i} \ne \overline{\gamma_{i}}$. Now $(u^{n}_{i} - u)|_{\partial M_{i} \setminus \overline{\gamma_{i}}} =0$. Since $\partial M_{i} \setminus \overline{\gamma_{i}}$ is a $C^{\infty}$ boundary portion of $M_{i}$, we also know $u^{n}_{i} - u \in C^{\infty} (M_{i} \setminus \overline{\gamma_{i}})$. (See the last paragraph of $\S$6.4 on \cite[p.~112]{Gilbarg_Trudinger}.)

Since $\mathrm{supp} \rho_{i} \subset M_{i} \setminus \overline{\gamma_{i}}$, by Lemma \ref{lem_subdomain_interior}, we know $\{  M_{i} \setminus \overline{\gamma_{i}}, M \setminus \mathrm{supp} \rho_{i} \}$ is an open cover of $M$. Since $\rho_{i} \in C^{k, \alpha} (M)$, we infer $\rho_{i} (u^{n}_{i} - u) \in C^{k, \alpha} (M)$. Hence
\begin{equation}\label{thm_high_global_2}
u^{n} -u = \sum_{i=1}^{m} \rho_{i} (u^{n}_{i} - u) \in C^{k, \alpha} (M).
\end{equation}
To estimate $\| u^{n} -u \|_{C^{k, \alpha} (M)}$, also by \eqref{thm_high_global_2}, it suffices to estimate the derivatives of $u^{n}_{i} -u$ up to order $k$ in a neighborhood of $\mathrm{supp} \rho_{i}$ for each $i$. We can cover $\mathrm{supp} \rho_{i}$ by open coordinate charts $W_{1}, \dots, W_{s}$ such that
\[
\bigcup_{j=1}^{s} K_{j} = \mathrm{supp} \rho_{i} \subset \bigcup_{j=1}^{s} W_{j} \subset M_{i} \setminus \overline{\gamma_{i}},
\]
where $K_{j} \subset \subset W_{j}$ for each $j$. It in turn suffices to estimate $u^{n}_{i} -u$ in a neighborhood of $K_{j}$ in $W_{j}$ for each $j$. Without loss of generality, we may assume $M_{i} \setminus \overline{\gamma_{i}}$ itself has coordinates $(x_{1}, \dots, x_{d})$, and estimate the usual partial derivatives of $u^{n}_{i} -u$ with respect to $(x_{1}, \dots, x_{d})$.

We claim there are an open neighborhood $V_{k}$ of $\mathrm{supp} \rho_{i}$ and a constant $C_{i}$ such that
\begin{equation}\label{thm_high_global_3}
\| u^{n}_{i} - u \|_{C^{k,\alpha}(V_{k})} \le C_{i} \| u^{n}_{i} - u \|_{C^{0}(M_{i})},
\end{equation}
where $\overline{V_{k}} \subset M_{i} \setminus \overline{\gamma_{i}}$, and $C_{i}$ only depends on $i$, $k$, $\alpha$ and $V_{k}$. Since $\mathfrak{L} (u^{n}_{i} -u) =0$, this follows from \cite[Problem~6.1]{Gilbarg_Trudinger} when $\partial M_{i} = \overline{\gamma_{i}}$. (Note that, by \cite{Gilbarg_Trudinger}, this is true when $\alpha >0$. It certainly also holds for $C^{k,0}$ as $\| \cdot \|_{C^{k,0}} \le \| \cdot \|_{C^{k, \alpha}}$.)

When $\partial M_{i} \ne \overline{\gamma_{i}}$, we also have \eqref{thm_high_global_3}. Here are more details. We may identify $M_{i} \setminus \overline{\gamma_{i}}$ with a relatively open subset of the half space $\mathbb{R}^{d}_{+} := \{ (x_{1}, \dots, x_{d}) \in \mathbb{R}^{d} \mid x_{d} \ge 0 \}$, and $\partial (M_{i} \setminus \overline{\gamma_{i}})$ is contained in $\partial \mathbb{R}^{d}_{+} = \{ (x_{1}, \dots, x_{d}) \in \mathbb{R}^{d} \mid x_{d} = 0 \}$. By the fact $(u^{n}_{i} - u)|_{\partial (M_{i} \setminus \overline{\gamma_{i}})} =0$ and \cite[Lemma~6.4]{Gilbarg_Trudinger}, we see \eqref{thm_high_global_3} holds when $k \le 2$. Let's consider the case of $k=3$. By \eqref{eqn_operator_local}, we have
\begin{equation}\label{thm_high_global_4}
\mathfrak{L} (u^{n}_{i} -u) = -\sum_{s=1}^{d} \sum_{t=1}^{d} g^{st} \tfrac{\partial^{2}}{\partial x_{s} \partial x_{t}} (u^{n}_{i} -u) + \sum_{s=1}^{d} b^{s} \tfrac{\partial}{\partial x_{s}} (u^{n}_{i} -u) + c (u^{n}_{i} -u).
\end{equation}
Taking the partial derivative of \eqref{thm_high_global_4} with respect to $x_{r}$ for $1 \le r \le d-1$, we see $\widetilde{\mathfrak{L}} [\frac{\partial}{\partial x_{r}} (u^{n}_{i} -u)] = h$, where $\widetilde{\mathfrak{L}}$ is an elliptic operator of 2nd order with $C^{\infty}$ coefficients, and $h$ is a function involving the derivatives of $u^{n}_{i} -u$ up to order $2$ together with those of the above $g^{st}$ and $b^{s}$. Since $r \le d-1$, we have $[\frac{\partial}{\partial x_{r}} (u^{n}_{i} -u)]|_{\partial (M_{i} \setminus \overline{\gamma_{i}})} =0$. By \cite[Lemma~6.4]{Gilbarg_Trudinger} again, we have
\[
\| \tfrac{\partial}{\partial x_{r}} (u^{n}_{i} - u) \|_{C^{2,\alpha}(V_{3})} \le \widetilde{C}_{i} (\| \tfrac{\partial}{\partial x_{r}} (u^{n}_{i} - u) \|_{C^{0}(V_{2})} + \| h \|_{C^{0,\alpha}(V_{2})}),
\]
where $V_{3}$ is an open neighborhood of $\mathrm{supp} \rho_{i}$ such that $\overline{V_{3}} \subset V_{2}$, and $\widetilde{C}_{i}$ only depends on $i$, $\alpha$ and $V_{3}$. By the fact that \eqref{thm_high_global_3} holds for $k=2$, we infer $\| \frac{\partial}{\partial x_{r}} (u^{n}_{i} -u) \|_{C^{2, \alpha} (V_{3})}$ is dominated by $\| u^{n}_{i} - u \|_{C^{0}(M_{i})}$ for all $r \le d-1$. In other words, the norms of all partial derivatives $\| \tfrac{\partial^{3}}{\partial x_{s} \partial x_{t} \partial x_{r}} (u^{n}_{i} -u) \|_{C^{0,\alpha} (V_{3})}$ of order $3$, except $\| \tfrac{\partial^{3}}{\partial x_{d}^{3}} (u^{n}_{i} -u) \|_{C^{0,\alpha} (V_{3})}$, is dominated by $\| u^{n}_{i} - u \|_{C^{0}(M_{i})}$. On the other hand, taking the partial derivative of \eqref{thm_high_global_4} with respect to $x_{d}$, we see $\tfrac{\partial^{3}}{\partial x_{d}^{3}} (u^{n}_{i} -u)$ can be expressed by other partial derivatives of $u^{n}_{i} -u$ of orders no more than $3$. So $\| \tfrac{\partial^{3}}{\partial x_{d}^{3}} (u^{n}_{i} -u) \|_{C^{0,\alpha} (V_{3})}$ is dominated as well. In summary, \eqref{thm_high_global_3} is verified when $k=3$. We obtain \eqref{thm_high_global_3} in general by induction on $k$.

By \eqref{thm_high_global_3} and the fact that $\rho_{i} \in C^{k,\alpha} (M)$ for all $i$, we get \eqref{thm_high_global_1}. It's necessary to point out that the $C$ in \eqref{thm_high_global_1} only depends on $k$, $\alpha$, $\mathrm{supp} \rho_{i}$ and $\| \rho_{i} \|_{C^{k,\alpha} (M)}$ for all $i$. So this $C$ is independent of $n$, $u$, and $u^{0}$.
\end{proof}

Finally, we prove Theorem \ref{thm_high_local}.
\begin{proof}[Proof of Theorem \ref{thm_high_local}]
(1). By Theorem \ref{thm_high_global}, we have $u^{n-1} - u \in C^{0,1} (M)$. By Rademacher's theorem (see Theorem 5 in \cite[\S4.2.3]{evans_gariepy}), this is equivalent to say $u^{n-1} - u \in W^{1,\infty} (M)$. Thus, for each $i$, there exists a solution $v_{i} \in H^{1} (M_{i})$ to
\begin{equation}\label{thm_high_local_3}
\left\{
\begin{aligned}
\mathfrak{L} v_{i}  & = 0, & \text{in} \ M_{i} \setminus \partial M_{i};
\\
v_{i}  & = u^{n-1} -u , & \text{on} \ \partial M_{i}.
\end{aligned}
\right.
\end{equation}
Since $\partial M_{i}$ is Lipschitz, by \cite[Corollary~8.28]{Gilbarg_Trudinger}, the solution $v_{i} \in C^{0} (M_{i})$. We know $u^{n}_{i} -u \in C^{0} (M_{i})$ is also a solution to \eqref{thm_high_local_3}. By the uniqueness in Lemma \ref{lem_algorithm_wellpose}, we infer $u^{n}_{i} -u = v_{i} \in H^{1} (M_{i})$.

We first estimate $\| u^{n}_{i} -u \|_{L^{2}(M_{i})}$. Since $u^{n}_{i} -u$ solves \eqref{thm_high_global_3}, by the weak maximum principle \ref{cor_maximum}, we have
\begin{align}\label{thm_high_local_4}
\| u^{n}_{i} -u \|_{L^{2}(M_{i})} & \le |M_{i}|^{\frac{1}{2}} \| u^{n}_{i} -u \|_{C^{0}(M_{i})} = |M_{i}|^{\frac{1}{2}} \| u^{n-1} -u \|_{C^{0}(\partial M_{i})} \nonumber \\
& \le |M_{i}|^{\frac{1}{2}} \| u^{n-1} -u \|_{C^{0} (M_{i})},
\end{align}
where $|M_{i}|$ is the volume of $M_{i}$.

Second, we estimate $\| \nabla (u^{n}_{i} -u) \|_{L^{2}(M_{i})}$. In the following, the constants $A_{j}$'s are all positive and independent of $n$, $u$, $u^{0}$, and the index $i$ of subdomains.

Since $u^{n}_{i}|_{\partial M_{i}} = u^{n-1}|_{\partial M_{i}}$, we have $u^{n}_{i} -u^{n-1} = (u^{n}_{i} -u) - (u^{n-1} - u) \in H^{1}_{0} (M_{i})$. Integrating by parts, we see
\begin{align}\label{thm_high_local_5}
& (\mathfrak{L} (u^{n}_{i} -u^{n-1}), u^{n}_{i} -u^{n-1})_{M_{i}} \nonumber \\
= & \int_{M_{i}} \langle \nabla (u^{n}_{i} -u^{n-1}), \nabla (u^{n}_{i} -u^{n-1}) \rangle  + \int_{M_{i}} \langle \vec{b}, \nabla (u^{n}_{i} -u^{n-1}) \rangle \cdot (u^{n}_{i} -u^{n-1}) \nonumber \\
& + \int_{M_{i}} c (u^{n}_{i} -u^{n-1})^{2},
\end{align}
where $(\cdot, \cdot)_{M_{i}}$ is the dual pairing between $H^{-1}(M_{i})$ and $H^{1}_{0}(M_{i})$. We also see
\[
\mathfrak{L} (u^{n}_{i} -u^{n-1})= \mathfrak{L} (u^{n}_{i} -u) + \mathfrak{L} (u - u^{n-1}) = \mathfrak{L} (u - u^{n-1})
\]
because $\mathfrak{L} (u^{n}_{i} -u) =0$. By \eqref{thm_high_local_5} and the Poincar\'{e} inequality, also by the fact $c \ge 0$ and $\vec{b}$ is bounded in \eqref{eqn_operator}, there exist positive constants $A_{1}$ and $A_{2}$ such that
\begin{align*}
& \| \nabla (u^{n}_{i} -u^{n-1}) \|^{2}_{L^{2}(M_{i})} - A_{1} \| \nabla (u^{n}_{i} -u^{n-1}) \|_{L^{2}(M_{i})} \cdot \| u^{n}_{i} -u^{n-1} \|_{L^{2}(M_{i})} \\
\le & (\mathfrak{L} (u^{n}_{i} -u^{n-1}), u^{n}_{i} -u^{n-1})_{M_{i}} =  (\mathfrak{L} (u - u^{n-1}), u^{n}_{i} -u^{n-1})_{M_{i}} \\
\le & \| \mathfrak{L} (u^{n-1} -u) \|_{H^{-1}(M_{i})} \cdot \| u^{n}_{i} -u^{n-1} \|_{H^{1}(M_{i})} \\
\le & A_{2} \| u^{n-1} -u \|_{H^{1}(M_{i})} \cdot \| \nabla (u^{n}_{i} -u^{n-1}) \|_{L^{2}(M_{i})},
\end{align*}
Thus
\begin{align*}
& \| \nabla (u^{n}_{i} -u^{n-1}) \|_{L^{2}(M_{i})} \leq A_{1} \| u^{n}_{i} -u^{n-1} \|_{L^{2}(M_{i})} +  A_{2} \| u^{n-1} -u \|_{H^{1}(M_{i})} \\
\le & A_{1} \| u^{n}_{i} -u \|_{L^{2}(M_{i})} + A_{1} \| u^{n-1} -u \|_{L^{2}(M_{i})} +  A_{2} \| u^{n-1} -u \|_{H^{1}(M_{i})}.
\end{align*}
Therefore,
\begin{align}\label{thm_high_local_6}
& \| \nabla (u^{n}_{i} -u) \|_{L^{2}(M_{i})} \le \| \nabla (u^{n}_{i} -u^{n-1}) \|_{L^{2}(M_{i})} + \| \nabla (u^{n-1} -u) \|_{L^{2}(M_{i})} \nonumber \\
\le & A_{1} \| u^{n}_{i} -u \|_{L^{2}(M_{i})} + (1+ A_{1} + A_{2}) \| u^{n-1} -u \|_{H^{1}(M_{i})}.
\end{align}
Combining \eqref{thm_high_local_4} with \eqref{thm_high_local_6}, we obtain
\[
\| u^{n}_{i} -u \|_{H^{1}(M_{i})} \le A_{3} \| u^{n-1} -u \|_{C^{0} (M_{i})} + A_{4} \| u^{n-1} -u \|_{H^{1}(M_{i})}.
\]

Note that $\| \cdot \|_{H^{1} (M_{i})}$ is dominated by $\| \cdot \|_{W^{1,\infty} (M_{i})}$. By Rademacher's theorem again, this is further dominated by $\| \cdot \|_{C^{0,1} (M_{i})}$. As a result,
\[
\| u^{n}_{i} -u \|_{H^{1}(M_{i})} \le A_{5} \| u^{n-1} -u \|_{C^{0,1} (M_{i})}.
\]
Now \eqref{thm_high_local_1} follows from \eqref{thm_high_global_1}.

(2). We know $u^{n}_{j} -u$ is a solution to \eqref{thm_high_local_3} with $i=j$. By Theorem \ref{thm_high_global}, we have $u^{n-1} -u \in C^{k, \alpha}$. Since $\partial M_{j}$ is $C^{k, \alpha}$, by \cite[Theorem~6.19~\&~Problem~6.2]{Gilbarg_Trudinger}, we infer $u^{n}_{j} - u \in C^{k,\alpha}(M_{j})$ and
\[
\| u^{n}_{j} - u \|_{C^{k,\alpha}(M_{j})} \le \widetilde{C}_{j} (\| u^{n}_{j} - u \|_{C^{0}(M_{j})} + \| u^{n-1} - u \|_{C^{k,\alpha}(M_{j})}),
\]
where $\widetilde{C}_{j}$ is independent of $n$, $u$ and $u^{0}$. Again, since $u^{n}_{j} -u$ is a solution to \eqref{thm_high_local_3}, by the weak maximum principle \ref{cor_maximum}, we see
\[
\| u^{n}_{j} - u \|_{C^{0}(M_{j})} = \| u^{n}_{j} - u \|_{C^{0}(\partial M_{j})} = \| u^{n-1} - u \|_{C^{0}(\partial M_{j})} \le \| u^{n-1} - u \|_{C^{k,\alpha}(M_{j})}
\]
and hence
\[
\| u^{n}_{j} - u \|_{C^{k,\alpha}(M_{j})} \le (1 + \widetilde{C}_{j}) \| u^{n-1} - u \|_{C^{k,\alpha}(M_{j})}.
\]
Now \eqref{thm_high_local_2} follows from \eqref{thm_high_global_1}.
\end{proof}

\section{Generalization}\label{sec_general}
We discuss some generalizations of the main results of this paper.

\subsection{Convergence in General}\label{subsec_converge_general}
As mentioned before, theorems in $\S$\ref{sec_main} are insufficient, under Assumptions \ref{asp_wellpose} and \ref{asp_decomposition} alone, to establish the geometric convergence of Algorithm \ref{alg_continuous_parallel}. However, we can indeed prove the general Theorem \ref{thm_converge} below ensuring the geometric convergence in full generality. Before formulating this theorem, we need to introduce a new type of winding number which differs slightly from the one in Definition \ref{def_winding}.

Let $c$ be the function in \eqref{eqn_operator}. Define $\widetilde{I}_{0} := \emptyset$ and $\widetilde{I}_{1} := I_{1} \cup P(c)$, where $I_{1}$ is the one defined prior to Definition \ref{def_winding} and
\[
P(c) := \{ i \mid \text{$c>0$ somewhere in $M_{i}$} \}.
\]
We first point out that $\widetilde{I}_{1} \ne \emptyset$ under Assumption \ref{asp_wellpose}. In fact, for the (1) of Assumption \ref{asp_wellpose}, we see $P(c) \ne \emptyset$; for the (2) of Assumption \ref{asp_wellpose}, by Lemma \ref{lem_winding_start}, we have $I_{1} \ne \emptyset$. (Note that $I_{1} = \emptyset$ if $\partial M = \emptyset$.) When $n>1$, we define inductively
\[
\widetilde{I}_{n} := \widetilde{I}_{n-1} \cup \left\{ i \middle| 1 \le i \le m, \partial M_{i} \cap \left( \bigcup_{j \in \widetilde{I}_{n-1}} U_{j} \right) \ne \emptyset \right\}.
\]
Clearly, $\widetilde{I}_{n} \subseteq \widetilde{I}_{n+1}$ for each $n$. If $\widetilde{I}_{n} = \widetilde{I}_{n+1}$, then $\widetilde{I}_{n} = \widetilde{I}_{k}$ for all $k>n$. There exists an $\widetilde{N} \ge 1$ such that $\widetilde{I}_{\widetilde{N}-1} \ne \widetilde{I}_{\widetilde{N}} = \widetilde{I}_{\widetilde{N}+1}$.

\begin{definition}\label{def_general_winding}
We call the above $\widetilde{N}$ the \emph{winding number} of $(c, \mathcal{D}, \rho)$. We also let $\widetilde{N}(c, \mathcal{D}, \rho)$ denote it.
\end{definition}

There are two obvious differences between $\widetilde{N}(c, \mathcal{D}, \rho)$ and $N(\mathcal{D}, \rho)$. First, $N(\mathcal{D}, \rho)$ only depends on $\mathcal{D}$ and $\rho$, but $\widetilde{N}(c, \mathcal{D}, \rho)$ depends on $c$ further. Second, in practice, $N(\mathcal{D}, \rho)$ is useful only when $\partial M \ne \emptyset$, but $\widetilde{N}(c, \mathcal{D}, \rho)$ has no such shortcoming. If $\partial M \ne \emptyset$, we can prove easily $\widetilde{I}_{n} \supseteq I_{n}$ for all $n$ by induction on $n$. Similar to Proposition \ref{prop_winding}, we have the following result.
\begin{proposition}\label{prop_general_winding}
$\widetilde{I}_{\widetilde{N}} = \{ i \mid 1 \le i \le m \}$.
\end{proposition}
\begin{proof}[Sketch of Proof]
If $\partial M \ne \emptyset$, then $\widetilde{I}_{\widetilde{N}} \supseteq \widetilde{I}_{N} \supseteq I_{N}$. The conclusion follows from Proposition \ref{prop_winding}. If $\partial M = \emptyset$, the conclusion is proved by an argument similar to the proof of Proposition \ref{prop_winding}. Particularly, we also define $D_{n}$ and obtain $D_{n_{0}} = M$ as we did in that proof. By the (1) of Assumption \ref{asp_wellpose}, there exists a $j_{1} \in J_{n_{1}}$ such that $j_{1} \in P(c) = \widetilde{I}_{1}$ and $n_{1} \le n_{0}$. Other part of the proof duplicates that proof.
\end{proof}

We list various properties of $\widetilde{N}(c, \mathcal{D}, \rho)$ in some special cases. If $c$ satisfies the assumption of Corollary \ref{cor_rate_nondegenerate}, i.e.,~$c>0$ somewhere in $M_{i}$ for all $i$ satisfying $\partial M_{i} = \overline{\gamma_{i}}$, then $\widetilde{I}_{1} = \{ i \mid 1 \le i \le m \}$ and hence $\widetilde{N}(c, \mathcal{D}, \rho) =1$. If $\partial M \ne \emptyset$, then $\widetilde{I}_{N} \supseteq I_{N} = \{ i \mid 1 \le i \le m \}$ and hence $\widetilde{N}(c, \mathcal{D}, \rho) \le N(\mathcal{D}, \rho)$. If $\partial M \ne \emptyset$ and $c=0$, then $\widetilde{N}(c, \mathcal{D}, \rho) = N(\mathcal{D}, \rho)$.

We can prove the following general theorem describing the geometric convergence of Algorithm \ref{alg_continuous_parallel} under Assumptions \ref{asp_wellpose} and \ref{asp_decomposition} alone.
\begin{theorem}\label{thm_converge}
Let $\widetilde{N} (c, \mathcal{D}, \rho)$ be the winding number of $(c, \mathcal{D}, \rho)$ (see Definition \ref{def_general_winding}). Let $u^{0} \in C^{0} (M)$ satisfying $u^{0}|_{\partial M} = \varphi$ be an arbitrary initial guess of Algorithm \ref{alg_continuous_parallel}. Let $u^{n}$ and $u^{n}_{i}$ be the approximations generated by Algorithm \ref{alg_continuous_parallel}.

Then the following hold:
\begin{enumerate}[(1)]
\item $\| u^{n} - u \|_{C^{0} (M)} \le \max_{1 \le i \le m} \{ \| u^{n}_{i} - u \|_{C^{0} (M_{i})} \} \le \| u^{n-1} - u \|_{C^{0} (M)}$ for $n>0$;

\item there exists a constant $L \in (0,1)$ such that, $\forall n \ge 0$,
\[
\| u^{n+\widetilde{N}} - u \|_{C^{0} (M)} \le L^{\widetilde{N}} \| u^{n} - u \|_{C^{0} (M)},
\]
where $L$ is independent of $n$, $u^{0}$ and $u$;

\item there exists a constant $C>0$ such that $\forall n \ge 0$,
\[
\| u^{n} - u \|_{C^{0} (M)} \le C L^{n} \| u^{0} - u \|_{C^{0} (M)},
\]
where $L$ is the one in (2), and $C$ is independent of $n$, $u^{0}$ and $u$;

\item if $\widetilde{N} >2$, then there exists a particular choice of $u^{0} \in C^{0} (M)$ with $u^{0}|_{\partial M} = \varphi$ such that $u^{0} \ne u$ and
\begin{equation}\label{thm_converge_1}
\| u^{n} - u \|_{C^{0} (M)} = \| u^{0} - u \|_{C^{0} (M)} \qquad \text{for $n<\widetilde{N} -2$};
\end{equation}
suppose additionally either (resp.~both) $M \ne \bigcup_{i \ne j} M_{i}$ for each $j$ or (resp.~and) $M_{i} \cap \partial M = \emptyset$ for every $i$ satisfying $\partial M_{i} = \overline{\gamma_{i}}$, then \eqref{thm_converge_1} also holds for $n < \widetilde{N} -1$ (resp.~$n < \widetilde{N}$).
\end{enumerate}
\end{theorem}
\begin{proof}[Sketch of Proof]
(1). This is (1) of Theorem \ref{thm_rate}.

(2). We study the $\tau$ in Lemma \ref{lem_bound_function}. Note that $\tau \in C^{0}_{0} (M)$, where $C^{0}_{0} (M) = C^{0} (M)$ if $\partial M = \emptyset$. Now $\tau <1$ on $\bigcup_{i \in \widetilde{I}_{1}} U_{i}$ because $\theta_{i}|_{U_{i}} <1$ when $i \in \widetilde{I}_{1}$. Following the proof of Proposition \ref{prop_iteration_function}, but with Proposition \ref{prop_general_winding} in place of Proposition \ref{prop_winding}, we obtain $0 \le T^{\widetilde{N} -1} \tau <1$. Other part of the proof duplicates that of (1) of Theorem \ref{thm_rate_boundary}.

(3). This follows immediately from (2). (See the proof of (2) of Theorem \ref{thm_rate_boundary}.)

(4). The proof is similar to that of Theorem \ref{thm_slow} (see also Remark \ref{rmk_slow}). Choose a $v \in C^{0}_{0} (M)$ satisfying $0 \le v \le 1$ and $v=1$ on $\bigcup_{i \notin \widetilde{I}_{2}} M_{i}$ (or on $\bigcup_{i \notin \widetilde{I}_{1}} M_{i}$ if $M_{i} \cap \partial M = \emptyset$ for every $i$ with $\partial M_{i} = \overline{\gamma_{i}}$). Then $u^{0} := u+v$ is the desired one.
\end{proof}

The (3) of Theorem \ref{thm_converge} already shows the geometric convergence of Algorithm \ref{alg_continuous_parallel} in full generality. If $\partial M \ne \emptyset$, then $\widetilde{N} (c, \mathcal{D}, \rho) \le N(\mathcal{D}, \rho)$ and the (1)-(3) of Theorem \ref{thm_converge} imply Theorem \ref{thm_rate_boundary}. If $c$ satisfies the assumption of Corollary \ref{cor_rate_nondegenerate}, then $\widetilde{N}=1$ and the (2) of Theorem \ref{thm_converge} implies Corollary \ref{cor_rate_nondegenerate}. Furthermore, if $M$ and $c$ satisfies the assumption of Theorem \ref{thm_slow}, then $\widetilde{N} (c, \mathcal{D}, \rho) = N (\mathcal{D}, \rho)$ and the (4) of Theorem \ref{thm_converge} degenerates to Theorem \ref{thm_slow}.

\subsection{Equations of General Type}\label{subsec_type}
We show that the linear elliptic operator \eqref{eqn_operator} indeed covers all second-order linear elliptic operators on manifolds. More precisely, an operator of general form on a Riemannian manifold $(M,g)$ can be reduced to the form of \eqref{eqn_operator} for some Riemannian metric $\tilde{g}$ which may differ from $g$. If $M$ is compact, as noted in $\S$\ref{subsec_function}, then both $g$ and $\tilde{g}$ yield the same uniform structures on the usual function spaces such as $C^{k, \alpha} (M)$ and $W^{k,p} (M)$. Consequently, the theory developed in this paper applies to the full class of second-order linear elliptic equations on manifolds.

Suppose $A$ is a $C^{\infty}$, symmetric and positive definite $(2,0)$-tensor field on $M$. Recall that a $(2,0)$-tensor field $A$ assigns to each $x \in M$ an $A(x) \in T_{x} M \otimes T_{x} M$, where $T_{x} M$ is the tangent space of $M$ at $x$. In local coordinates $(x_{1}, \dots, x_{d})$,
\[
A(x) = \sum_{i=1}^{d} \sum_{j=1}^{d} a^{ij} (x) \frac{\partial}{\partial x_{i}} \otimes \frac{\partial}{\partial x_{j}}.
\]
By $C^{\infty}$, we mean $a^{ij} (x)$ are $C^{\infty}$ with respect to $x$. By symmetric and positive definite, we mean $A(x)$ is a symmetric and positive definite bilinear form on $T^{*}_{x} (M)$, where $T^{*}_{x} (M)$ is the cotangent space of $M$ at $x$; equivalently, $(a^{ij} (x))_{d \times d}$ is a symmetric and positive definite matrix.

A general linear elliptic operator of second order takes the form (cf.~\cite[p.~83]{Aubin})
\begin{equation}\label{eqn_operator_general}
\mathfrak{L} u := - (A, \nabla^{2} u) + \langle \vec{b}, \nabla u \rangle_{g} + cu,
\end{equation}
where $(\cdot, \cdot)$ is the pointwise dual pairing between $T_{x} M \otimes T_{x} M$ and $T^{*}_{x} M \otimes T^{*}_{x} M$, $\nabla$ is the covariant differential associated with $g$, and $\langle \cdot, \cdot \rangle_{g}$ is the Riemannian metric $g$. In local coordinates, we have
\[
- (A, \nabla^{2} u) = - \sum_{i=1}^{d} \sum_{j=1}^{d} a^{ij} \frac{\partial^{2} u}{\partial x_{i} \partial x_{j}} + \sum_{i=1}^{d} \sum_{j=1}^{d} \sum_{k=1}^{d} a^{ij} \Gamma_{ij}^{k} \frac{\partial u}{\partial x_{k}},
\]
where $\Gamma_{ij}^{k}$ are the Christoffel symbols (see \cite[p.~3]{Aubin}).

Via $g$, we may identify $T_{x} M$ with $T^{*}_{x} M$ by the Riesz representation theorem (cf.~\cite[p.~11]{conway}). As a result, $T_{x} M \otimes T^{*}_{x} M$, $T_{x} M \otimes T_{x} M$ and $T^{*}_{x} M \otimes T^{*}_{x} M$ are mutually identified. Assigning the identity map $I(x) \colon T_{x} M \rightarrow T_{x} M$ to each $x$; this defines a $C^{\infty}$ $(1,1)$-tensor field $I$, which is also identified with a $(2,0)$-tensor field $I_{g}$. Now
\[
\Delta_{g} u = (I_{g}, \nabla^{2} u),
\]
where $\Delta_{g} u$ is the Laplace operator associated with $g$. It's easy to see $I_{g}$ is symmetric and positive definite, so $\Delta_{g} u$ is a special case of $(A, \nabla^{2} u)$.

Let's construct a new Riemannian metric $\tilde{g}$ on $M$. Since $A$ is symmetric and positive definite, it defines an inner product on $T^{*}_{x} M$ at each $x$. By the Riesz representation again, this induces an inner product $\tilde{g}_{x}$ on $T_{x} M$. Since $A$ is $C^{\infty}$, the resulting $(0,2)$-tensor field $\tilde{g}$ is a $C^{\infty}$ Riemannian metric, and we have $A= I_{\tilde{g}}$. Here $I_{\tilde{g}}$ is defined analogously to $I_{g}$ but with respect to $\tilde{g}$. Then
\[
- (A, \nabla^{2} u) = - (I_{\tilde{g}}, \nabla^{2} u).
\]
Note that $\nabla$ is the covariant differential associated with $g$, not with $\tilde{g}$. Let $\widetilde{\nabla}$ be the one associated with $\tilde{g}$. For $C^{\infty}$ vector fields $X$ and $Y$,
\[
\nabla^{2} u (X,Y) = YXu - (\nabla_{Y} X) u, \qquad \widetilde{\nabla}^{2} u (X,Y) = YXu - (\widetilde{\nabla}_{Y} X) u
\]
and hence
\[
(\nabla^{2} u - \widetilde{\nabla}^{2} u) (X,Y) = (\widetilde{\nabla}_{Y} X - \nabla_{Y} X) u.
\]
It's straightforward to check that the map $(X,Y) \mapsto (\widetilde{\nabla}_{Y} X - \nabla_{Y} X)$ defines a $C^{\infty}$ $(1,2)$-tensor field. Thus
\begin{align*}
-(A, \nabla^{2} u) & = -(A, \widetilde{\nabla}^{2} u) + (A, \widetilde{\nabla}^{2} u - \nabla^{2} u) = -(I_{\tilde{g}}, \widetilde{\nabla}^{2} u) + Zu \\
& = - \Delta_{\tilde{g}} u + Zu,
\end{align*}
where $Z$ is a $C^{\infty}$ vector field. So the $\mathfrak{L}$ in \eqref{eqn_operator_general} equals
\begin{equation}\label{eqn_operator_reduce}
\mathfrak{L} u = - \Delta_{\tilde{g}} u + \langle \vec{b} + Z, \widetilde{\nabla} u \rangle_{\tilde{g}} + cu.
\end{equation}
Thus the general operator \eqref{eqn_operator_general} reduces to the form \eqref{eqn_operator}, which is ostensibly, but not actually, more special.

It's necessary to point out that the $c$ in \eqref{eqn_operator_reduce} equals that in \eqref{eqn_operator_general}. Therefore, \eqref{eqn_operator_general} satisfies Assumption \ref{asp_wellpose} if and only if \eqref{eqn_operator_reduce} does. The entire theory of this paper carries over to general second-order linear elliptic equations on compact manifolds.

\subsection{Equations with Low Regularity}
We address another generalization of \eqref{eqn_operator}. Suppose now the coefficients in \eqref{eqn_operator} have regularities lower than $C^{\infty}$. For example, the Riemannian metric $g$, the vector field $\vec{b}$ and the function $c$ are of regularities $C^{k, \alpha}$ or $W^{k,p}$. With suitable modifications, the principal results of this paper remain valid, provided the coefficients are not excessively singular.

In the low-regularity setting, the interpretation of \eqref{eqn_problem} requires $u$ to be more regular than an arbitrary distribution. Likewise, \eqref{alg_continuous_parallel_1} often requires the $\rho_{i}$ in Assumption \ref{asp_decomposition} to be more regular than merely continuous. For instance, if $g$ is $C^{1}$, and both $\vec{b}$ and $c$ are $C^{0}$, then it's safe to require $u \in H^{1} (M)$ and $\rho_{i} \in C^{0,1} (M)$. Furthermore, the maximum principle \ref{thm_maximum} plays a central role in our argument, and it relies on the continuity of the involved functions. So it's good to assume both $u$ and $u^{0}$ are in $C^{0} (M) \cap H^{1} (M)$. By global estimate on continuity (\cite[Corollary~8.28]{Gilbarg_Trudinger}) and the maximum principle for weak solutions (see \cite[$\S$8.1~\&~8.7]{Gilbarg_Trudinger}), these assumptions already suffice to establish all the conclusions of the theorems in this paper (including Theorem \ref{thm_converge}), with the exception of Theorems \ref{thm_high_global} and \ref{thm_high_local} and their corollaries. Note also that, in this setting, if we drop the continuity assumption on $u$ and $u^{0}$ while retaining the other hypotheses, the local H\"{o}lder estimates (\cite[Theorems~8.22~\&~8.27]{Gilbarg_Trudinger}) still imply the continuity of $u^{n} - u$ (for $n>0$) and $u^{n}_{i} - u$ (for $n>1$), as well as their geometric convergence in $C^{0}$, although the continuity of $u^{n}$ and $u^{n}_{i}$ themselves is not guaranteed.

The significant difference lies in Theorems \ref{thm_high_global} and \ref{thm_high_local} and their corollaries. For example, even if $\rho_{i} \in C^{\infty} (M)$, we can no longer conclude $(u^{n} - u) \in C^{\infty} (M)$. Nevertheless, one can still obtain convergence of derivatives in various senses--ranging from weak to strong--up to a certain order. The admissible senses of convergence, as well as the maximal order of differentiability, are determined by the regularities of \eqref{eqn_operator}, $\rho_{i}$, and $\partial M_{i}$. The proof is analogous to those of Theorems \ref{thm_high_global} and \ref{thm_high_local} and relies on the classical regularity theory of second-order linear elliptic equations (see e.g.,~Chapters 6, 8 and 9 in \cite{Gilbarg_Trudinger}).

\section{Conclusion}\label{sec_conclusion}
This paper presents an in-depth convergence theory for the continuous DDM, Algorithm \ref{alg_continuous_parallel}, for solving the elliptic boundary value problem \eqref{eqn_problem} on compact Riemannian manifolds. As shown in $\S$\ref{subsec_type}, the problem \eqref{eqn_problem} indeed covers the entire class of second-order linear elliptic equations on manifolds.

We outline a proof of the general Theorem \ref{thm_converge}, which thoroughly characterizes the convergence of Algorithm \ref{alg_continuous_parallel} in $C^{0}$-norm under Assumptions \ref{asp_wellpose} and \ref{asp_decomposition}, establishing in particular its geometric convergence. Some special cases of Theorem \ref{thm_converge} are focused; in these cases, Theorem \ref{thm_converge} splits into three results: Corollary \ref{cor_rate_nondegenerate}, Theorems \ref{thm_rate_boundary} and \ref{thm_slow}. We prove the three results in full detail. Although Theorem \ref{thm_converge} is technically more involved, its proof idea is parallel to that of the special cases. Readers who are acquainted with the special cases will have no difficulty in filling in the general details along the lines sketched in $\S$\ref{subsec_converge_general}.

We also derive several estimates on the convergence rate in $C^{0}$-norm, which are collected in Theorems \ref{thm_bound}, \ref{thm_comparison}, and \ref{thm_refine}. These results reflect how the geometric aspect of the domain decomposition and the equation itself affect the convergence.

Finally, Theorems \ref{thm_high_global} and \ref{thm_high_local} demonstrate that $C^{0}$-convergence implies convergence of high-order derivatives under suitable regularity assumptions.

\section*{Acknowledgements}
I thank Yiyan Xu for various discussions. This work was partially supported by NSFC 11871272.


\end{document}